\documentclass[11pt, a4paper]{amsart}
\usepackage{a4wide}
\usepackage{subfigure}
\usepackage{mathtools}
\usepackage{amssymb}
\usepackage{enumitem}
\usepackage[pdftex,
pdftitle={Interpolation of noise in SPDE approximations},
bookmarksopen,
colorlinks,
linkcolor=black,
urlcolor=black,
citecolor=black
]{hyperref}
\hypersetup{pdfauthor={G. Lord and A. Petersson}}
\usepackage{xcolor} 

\DeclareSymbolFontAlphabet{\amsmathbb}{AMSb}%

\newcommand{\R}{\amsmathbb{R}}
\newcommand{\C}{\amsmathbb{C}}
\newcommand{\cA}{\mathcal{A}}
\newcommand{\cD}{\mathcal{D}} 
\newcommand{\cE}{\mathcal{E}} 
\DeclareMathOperator{\Cov}{\mathsf{Cov}}
\newcommand{\dd}{\,\mathrm{d}}
\newcommand{\cT}{\mathcal{T}}
\newcommand{\cW}{\mathcal{W}}
\newcommand{\inpro}[3][{}]{ \langle #2 , #3 \rangle_{#1} }

\DeclareMathOperator{\E}{\amsmathbb{E}}
\newcommand{\norm}[2][{}]{\| #2 \|_{#1}}
\newcommand{\cS}{\mathcal{S}}
\newcommand{\dom}{\mathrm{dom}}
\newcommand{\N}{\amsmathbb{N}}
\newcommand{\cL}{\mathcal{L}}
\newcommand{\cN}{\mathcal{N}}
\newcommand{\cC}{\mathcal{C}}

\newcommand{\cO}{\mathcal{O}}
\newcommand{\cP}{\mathcal{P}}

\newcommand{\seminorm}[2][{}]{| #2 |_{#1}}
\newcommand{\cF}{\mathcal{F}}
\newcommand{\IP}{\amsmathbb{P}}
\DeclareMathOperator*{\esssup}{ess\,sup}

\newcommand{\Bignorm}[2][{}]{\Big\| #2 \Big\|_{#1}}
\newcommand{\epel}{\stackrel{\epsilon}{=}}
\newcommand{\tr}{\mathrm{tr}}

\newtheorem{lemma}{Lemma}[section]
\newtheorem{proposition}[lemma]{Proposition}

\newtheorem{theorem}[lemma]{Theorem}

\theoremstyle{remark}
\newtheorem{remark}[lemma]{Remark}
\theoremstyle{definition}

\newtheorem{assumption}[lemma]{Assumption}
\newtheorem{example}[lemma]{Example}

\makeatletter
\@namedef{subjclassname@1991}{{\normalfont 2020} Mathematics Subject Classification}
\makeatother

\begin{document}
\title[Interpolation of noise in SPDE approximations]{Piecewise linear interpolation of noise in finite element approximations of parabolic SPDEs}
\author[G.~Lord]{Gabriel Lord} \address[Gabriel Lord]{\newline Department of Mathematics, IMAPP, Radboud University.
\newline Postbus 9010, 6500 GL Nijmegen, The Netherlands.} \email[]{gabriel.lord@ru.nl}
	
\author[A.~Petersson]{Andreas Petersson} 
\address[Andreas Petersson]{Department of Mathematics, Faculty of Technology, Linnaeus University. 351 95 V\"axj\"o, Sweden \newline and \newline Department of Mathematics, The Faculty of Mathematics and Natural Sciences, University of Oslo. 
\newline Postboks 1053, Blindern, 0316 Oslo, Norway.} \email[]{andreas.petersson@lnu.se}

\thanks{The work of A. Petersson was supported in part by the Research Council of Norway (RCN) through project no.\ 274410.}	



\begin{abstract}
Efficient simulation of stochastic partial differential equations (SPDE) on general domains requires noise discretization. This paper employs piecewise linear interpolation of noise in a fully discrete finite element approximation of a semilinear stochastic reaction-advection-diffusion equation on a convex polyhedral domain. The Gaussian noise is white in time, spatially correlated, and modeled as a standard cylindrical Wiener process on a reproducing kernel Hilbert space. This paper provides the first rigorous analysis of the resulting noise discretization error for a general spatial covariance kernel.
The kernel is assumed to be defined on a larger regular domain, allowing for sampling by the circulant embedding method. The error bound under mild kernel assumptions requires non-trivial techniques like Hilbert--Schmidt bounds on products of finite element interpolants, entropy numbers of fractional Sobolev space embeddings and an error bound for interpolants in fractional Sobolev norms. Examples with kernels encountered in applications are illustrated in numerical simulations using the FEniCS finite element software. Key findings include:  noise interpolation does not introduce additional errors for Mat\'ern kernels in $d\ge2$; there exist kernels that yield dominant interpolation errors; and generating noise on a coarser mesh does not always compromise accuracy.
\end{abstract}

\maketitle

\section{Introduction}

Finite element approximation excels at handling complex geometries. We unlock its potential for stochastic partial differential equations (SPDEs) by interpolating the driving noise on a convex polyhedron $\cD \subset \R^d, d \le 3,$ and analyzing the induced error. This applies to stochastic advection-reaction-diffusion equations
\begin{equation*}
	\frac{\partial X}{\partial t}(t,x) -\sum_{i,j = 1}^d \frac{\partial}{\partial x_j} \left(a_{i,j} \frac{\partial X}{\partial x_i}\right)(t,x) + c(x) X(t,x) = F(X(t))(x) + G(X(t)) \frac{\partial W(t,x)}{\partial t}
\end{equation*}
for $(t,x) \in (0,T] \times \cD$ with initial value $X(0,x) := X_0(x), x \in \cD$. The functions $(a_{i,j})_{i,j=1}^d, c$ satisfy an ellipticity condition, and homogeneous Dirichlet or Neumann boundary conditions are considered. The nonlinearities are given by $(G(u)v)(x) := g(u(x),x) \cdot v(x)$ and $F(u)(x) := b(x) \cdot (\nabla u)(x) + f(u(x),x)$ for functions $u,v$ on $\cD$, a vector field $b$ and Lipschitz functions $f,g$ on $\R \times \cD$. The noise $W$ is Gaussian, white in time with spatial covariance kernel $q$. These equations have numerous applications in fields like hydrodynamics, geophysics and chemistry and are treated as It\^o stochastic differential equations in the Hilbert space $H = L^2(\cD)$ of square integrable functions:
\begin{equation}
	\label{eq:spde}
	\begin{split}
		\dd X(t) + A X(t) &= F(X(t)) \dd t + G(X(t)) \dd W(t), \quad t \in (0,T], \\
		X(0) &= X_0,
	\end{split}
\end{equation}
see \cite{DPZ14}. Here $W$ is a standard cylindrical Wiener process on the reproducing kernel Hilbert space (RKHS) $H_q(\cD)$ (equivalently a $Q$-Wiener process for the associated integral operator $Q$, see Remark \ref{rem:Q-Wiener}) and $-A$ generates an analytical semigroup.

We discretize \eqref{eq:spde} using a semi-implicit Euler method combined with piecewise linear finite elements for a triangle mesh $\cT_h$ on $\cD$ with maximal mesh size $h \in (0,1]$. Such approximations are well-studied \cite{ANZ98, DZ02, W05, Y05, KLL10, BL12c, K14, CH19}, but rigorous analyses of \textit{noise discretizations} are rare. This refers to the approximation of
\begin{equation}
	\label{eq:stochastic-increment}
	\cW_k := \inpro[H]{G(X_{h,\Delta t}^j) \Delta W^j}{\varphi^h_k} = \int_{\cD} g(X_{h,\Delta t}^j(x),x) \Delta W^j(x) \varphi^h_k(x) \dd x
\end{equation}
for each time $t_j$ and nodal basis function $\varphi^h_k$. Here $\Delta W^j := W(t_{j+1}) - W(t_j)$ is independent of the approximation $X_{h,\Delta t}^j$ of $X(t_j)$ with constant time step $\Delta t := t_{j+1} - t_j$. Conditioned on $X_{h,\Delta t}^j$, the vector $(\cW_k)_{k = 1}^{N_h}$ (with $N_h$ the number of nodes in $\cT_h$) is therefore Gaussian with
\begin{equation*}
	\Cov(\cW_k,\cW_\ell) = \Delta t  \int_{\cD} \int_{\cD} g(X_{h,\Delta t}^j(x),x) g(X_{h,\Delta t}^j(y),y) \varphi^h_k(x) \varphi^h_\ell(y) q(x,y) \dd x \dd y.
\end{equation*}
Recomputing the covariance matrix and sampling from it at each time $t_j$ by, e.g., Cholesky decompositions becomes costly. To remedy this, $\Delta W^j$ in~\eqref{eq:stochastic-increment} is discretized. In this paper we interpolate $\Delta W^j$ between meshes and derive a strong error bound:
\begin{equation}
	\label{eq:intro-strong-error}
	\sup_{j} \E\left[\norm[H]{X_{h,\Delta t}^j - X(t_j)}^p\right]^{1/p} \le C \big(h^{1+r_1} + \Delta t^{1/2} + (h')^{r_2}\big), \quad \Delta t, h, h' \in (0,1]
\end{equation}
for $p \in [2,\infty)$, rates $r_1, r_2 \ge 0$, mesh sizes $h, h'$ and some $C<\infty$. Here $h^{1+r_1} + \Delta t^{1/2}$ is the standard SPDE approximation error and $(h')^{r_2}$ the noise discretization error.

Literature on noise discretization for parabolic SPDEs mainly focus on truncation of an eigenbasis of $Q$ \cite{BL12c}. Explicit bases are only available for special kernels $q$ and domains $\cD$, like cubes, which partly defeats the purpose of using finite elements.
Numerical approximation of the basis is an alternative, analyzed in~\cite{KLL10} for $F = 0, G = I$. The requirements on $q$ therein are hard to verify, but $q$ has to be piecewise analytic (a strong condition) for the cost of $(\cW_k)_{k = 1}^{N_h}$ to be $\cO(\log(N_h)N_h)$ \cite[Corollary~3.5]{KLL10}. In~\cite{KLL11}, a related approach of expanding the noise on wavelets is used, also for $F = 0, G = I$. Conditions are derived for which a severe truncation of the expansion does not affect the strong error in a semidiscretization. However, constructing the wavelet basis is challenging for general $\cD$ and needs strong assumptions on $q$ (see~\cite[Section~7.2, Theorem~7.1]{KLL11} but note that a Dirichlet condition on $q$ is missing).

In this paper, noise discretization is handled by first assuming $q$ to be positive semidefinite on another polyhedron $\cS \supseteq \cD$, typically a cube. Given $\cT_h$ on $\cD$, we interpolate the noise onto another triangulation $\cT_{h'}$ on $\cS$, then we restrict it to $\cD$ and finally we interpolate it onto $\cT_h$. 
In other words, $\Delta W^j$ in~\eqref{eq:stochastic-increment} is replaced with $I_h R_{\cS \to \cD} I_{h'} \Delta \tilde W^j$, where $I_h$ and $I_{h'}$ are the piecewise linear interpolants with respect to $\cT_h$ and $\cT_{h'}$, $R_{\cS \to \cD}$ is the linear operator restricting functions on $\cS$ to functions on $\cD$ and $\tilde W$ is an extension of $W$ with the same covariance. Figure~\ref{subfig:domains-arrow} illustrates this process for a dodecagon $\cD$ inside a square $\cS$. In terms of the basis functions $\varphi^{h'}_k$ for the $N_{h'}$ nodes $x'_k$ of $\cT_{h'}$,
\begin{equation}
	\label{eq:IhW-expansion}
	I_{h'} \Delta \tilde W^j = \sum_{k = 1}^{N_{h'}} \Delta \tilde W^j(x'_k) \varphi^{h'}_k \text{ with } \Cov(\Delta \tilde W^j(x'_k),\Delta \tilde W^j(x'_\ell)) = \Delta t q(x'_k,x'_\ell).
\end{equation}
We thus only need to sample $N_{h'}$ values from $q$ to approximate~\eqref{eq:stochastic-increment}, avoiding integrals involving $q$. If $q$ is stationary and we choose $\cS$ to allow for a uniform $\cT_{h'}$ the method becomes easy to implement in modern finite element software since fast exact FFT-based algorithms can be used to sample from $q$. These FFT algorithms include the circulant embedding method applied to a Matérn kernel, where the sampling cost of $(\cW_k)_{k = 1}^{N_h}$ is $\cO(\log(N_{h'})N_{h'})$\cite{GKNSS18}. Figure~\ref{subfig:realization} illustrates this with a realization of $X_{h,\Delta t}$ computed with this method with parameters as in Example~\ref{ex:numerical-matern}. If $q$ is non-stationary, one still needs to compute a Cholesky decomposition. In this setting, it should also be possible to combine the results of this paper with a truncation of a numerically computed eigenexpansion of $I_{h'} \Delta \tilde W^j$ for a more efficient algorithm.

\begin{figure}
	\centering
	\subfigure[Noise is sampled (interpolated) onto $\cT_{h'}$ on $\cS \supset \cD$, then restricted to $\cD$ and finally interpolated onto $\cT_{h}$ on $\cD$. \label{subfig:domains-arrow}]{\includegraphics[height=110px]{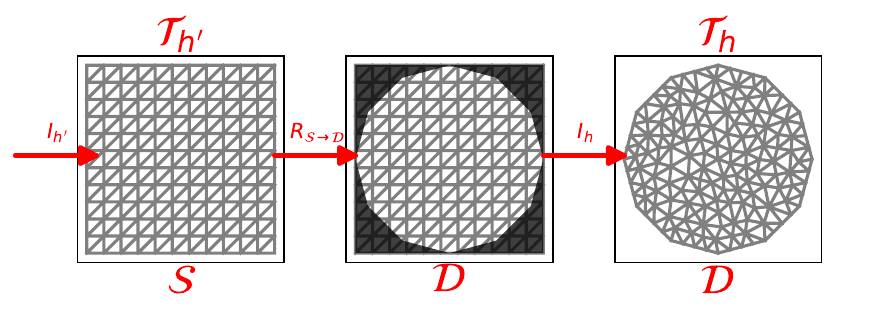}}
	\subfigure[A realization of $X_{h,\Delta t}$ at (left to right) $t_j=0.9, 0.95$ and $1.0$. \label{subfig:realization}]{\includegraphics[height=110px]{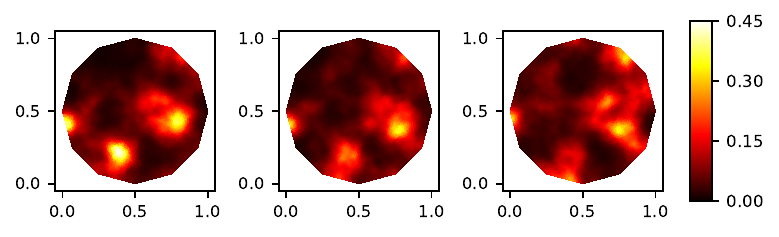}}		
	\caption{Noise discretization idea as in Example~\ref{ex:numerical-matern} for multiplicative noise, with $\cD$ a dodecagon.}
\end{figure}

We derive~\eqref{eq:intro-strong-error} under mild conditions on $q$, Assumption~\ref{assumption:q}. To use the results of~\cite{K14} for the SPDE approximation error bound, we construct $\tilde W$ as a standard cylindrical Wiener process on the RKHS $H_q(\cS)$ and derive new Hilbert--Schmidt bounds on multiplication operators and finite element interpolant products. Non-trivial techniques are needed for the noise discretization error, like entropy numbers of Sobolev space embeddings and generalized error bounds for interpolants. The analysis shows that the noise discretization error is non-dominant for most kernels in $d\ge2$, (including the Mat\'ern class), but dominant for some factorizable kernels in $d\ge 2$ and rough Mat\'ern kernels in $d=1$. Moreover, there are settings where $r_2>1+r_1$, so that we may take $N_{h'} \ll N_h$, increasing computational savings with retained accuracy. 

Interpolation of noise in finite element approximations of \eqref{eq:spde} with $\cS = \cD$ was considered in \cite{BGT15}, but without analyzing how the strong error is affected. Lemma~6.1 in \cite{K14} contains the only such analysis we are aware of. The author considers modified interpolants and spatial semidiscretization only, for $F = 0, G = I, \cD = (0,1), A = -\Delta$ with Dirichlet boundary conditions and $q$ given by a sine expansion. Then the convergence of the strong error is unaffected, but the author notes that
\begin{quote}
	it is only barely analyzed in the literature, how [the replacement of $\Delta W^j$ with $I_h \Delta W^j$] affects the order of convergence. \ldots We leave a rigorous justification in a more general situation
	as a subject to future research. \hfill \cite[p.\ 139]{K14}
\end{quote}
We give this justification in Section~\ref{sec:interpolation} and demonstrate it by numerical simulations in Section~\ref{sec:simulations} for concrete example kernels. Sections~\ref{sec:preliminaries}-\ref{sec:regularity} contain prerequisite material.

\section{Notation and mathematical background}
\label{sec:preliminaries}

This paper employs generic constants $C$. These vary between occurrences and are independent of discretization parameters. We write $a \lesssim b$ for $a,b \in \R$ to signify $a \le C b$ for such a constant $C$.

\subsection{Operator theory and entropy numbers}
\label{sec:operator-theory}

For Banach spaces $U$ and $V$ (all over $\R$ unless otherwise stated), we denote by $\cL(U,V)$ the linear and bounded operators from $U$ to $V$ and set $\cL(U) := \cL(U,U)$. We write $U \hookrightarrow V$ if $U \subset V$ and $\norm[\cL]{I_{U \hookrightarrow V}} := \norm[\cL(U,V)]{I_{U \hookrightarrow V}} < \infty$.

Let $B_U$ and $B_V$ denote the unit spheres in $U$ and $V$, respectively. Given $\Gamma \in \cL(U,V)$, its $n$th entropy number $\epsilon_n(\Gamma)$ is defined by
\begin{equation*}
	\inf\Big\{ \delta \ge 0 : \Gamma(B_U) \subset \cup_{j = 1}^m \{v_j + \delta B_V \} \text{ for some } v_1, \ldots, v_m \in V \text{ and } m \le 2^{n-1} \Big\}.
\end{equation*}
The decay of the entropy numbers is a measure of the compactness of $\Gamma$.
Let us write $\norm[\cE^\infty(U,V)]{\Gamma} := \norm[\cL(U,V)]{\Gamma}$ and $\norm[\cE^p(U,V)]{\Gamma} :=  \norm[\ell^p]{(\epsilon_j(\Gamma))_{j=1}^\infty}$, where $\ell^p$, $p \in (1,\infty)$, is the space of $p$-summable sequences.
Then a H\"older-type property holds for a further Banach space $E$ and operator $\Lambda \in \cL(V,E)$:
\begin{equation}
	\label{eq:entropy-holder}
	\norm[\cE^{p}(U,E)]{\Lambda \Gamma} \le 2^{1/p} \norm[\cE^{p_1}(V,E)]{\Lambda} \norm[\cE^{p_2}(U,V)]{\Gamma}, \hspace{0.5em} p_1,p_2 \in (1,\infty], 1/{p} = 1/{p_1} + 1/{p_2}.
\end{equation}
This is essentially \cite[Theorem~14.3.3]{P80}, the constant in the bound can be found by following the proof of \cite[Theorem~14.2.3]{P80} (see also \cite[Theorems~12.1.4-5]{P80}). 

For separable Hilbert spaces $U$ and $H$, $\cL_2(U,H)$ is the space of Hilbert--Schmidt operators. It is an operator ideal and a separable Hilbert space with inner product 
\begin{equation*}
	\inpro[\cL_2(U,H)]{\Gamma_1}{\Gamma_2} := \sum_{j=1}^\infty \inpro[H]{\Gamma_1 e_j}{\Gamma_2 e_j}, \quad \Gamma_1, \Gamma_2 \in \cL_2(U,H), 
\end{equation*}
where $(e_j)_{j=1}^\infty$ is an arbitrary orthonormal basis (ONB) of $U$. We use the decay of entropy numbers to bound the norm induced by this inner product. Specifically, for a given operator $\Gamma \in \cL(U,H)$ it holds that 
\begin{equation}
	\label{eq:HS-entropy}
	\norm[\cL_2(U,H)]{\Gamma} \le 2 \norm[\cE^2(U,H)]{\Gamma}.
\end{equation}
This follows from the fact that the $n$th singular value of $\Gamma$ can be bounded from above by $2 \epsilon_n(\Gamma)$ for all $n \in \N$. This is \cite[Theorem~12.3.1]{P80} if one keeps in mind that all so called $s$-numbers on Hilbert spaces coincide (which in turn is \cite[Theorem~11.3.4]{P80}). Since $\norm[\cL_2(U,H)]{\Gamma}$ is the square root of the sum of squared singular values of $\Gamma$ (see \cite[Theorem~15.5.5]{P80}), we obtain~\eqref{eq:HS-entropy}. For more on entropy numbers, we refer to \cite{P80}.

For $\Gamma \in \cL(U,H)$, the range $\Gamma(U)$ is a Hilbert space with $\inpro[\Gamma(U)]{\cdot}{\cdot} := \inpro[U]{\Gamma^{-1}\cdot}{\Gamma^{-1}\cdot}$, where $\Gamma^{-1} := (\Gamma|_{\ker(\Gamma)^\perp})^{-1} \colon \Gamma(U) \to \ker(\Gamma)^\perp \subset U$ is the pseudoinverse and $\ker(\Gamma)^\perp$ the orthogonal complement of the kernel of $\Gamma$. The norm on $\Gamma(U)$ may also be represented by $\norm[\Gamma(U)]{u} := \min_{v \in U, \Gamma v = u} \norm[U]{v}$, 
see \cite[Remark~C.0.2]{PR07}.
Note that $\Gamma^{-1} \Gamma$ coincides with the orthogonal projection $P_{\ker(\Gamma)^\perp}$, whereas $\Gamma \Gamma^{-1}= I$. We write $\Sigma^+(H) \subset \cL(H)$ for the subspace of positive semidefinite operators on $H$ and refer to \cite[Appendix~C]{PR07} for more details on pseudoinverses.

\subsection{Reproducing kernel Hilbert spaces}
\label{sec:rkhs}

Let $q \colon \cD \times \cD \to \R$ be a positive semidefinite kernel on a non-empty connected subset $\cD \subset \R^d$, $d \in \N$. The RKHS $H_q(\cD)$ of $q$ on $\cD$ is the Hilbert space of functions on $\cD$ such that 
\begin{enumerate}[label=(\roman*)]
	\item 
	$q(x,\cdot) \in H_q(\cD)$ for all $x \in \cD$ and
	\item for all $f \in H_q(\cD), x \in \cD$, $f(x)=\inpro[H_q(\cD)]{f}{q(x,\cdot)}$ (the reproducing property).
\end{enumerate}
Each $q$ has a unique RKHS $(H_q(\cD),\inpro[H_q(\cD)]{\cdot}{\cdot})$, but $H_q(\cD)$ may have multiple equivalent norms for different kernels. A Hilbert space $H$ of functions on $\cD$ is an RKHS if and only if the evaluation functional $\delta_x \colon H \to \R$, $\delta_x f := f(x)$, is bounded for all $x \in \cD$ \cite[Theorem~10.2]{W04}. We write $H_q :=H_q(\cD)$ if $\cD$ need not be emphasized. Suppose that $q$ is positive semidefinite also on $\cS \supset \cD$ and write $R_{\cS \to \cD}$ for the linear operator that restricts a function $f$ on $\cS$ to a function on $\cD$. That is to say,
\begin{equation}
	\label{eq:restriction-def}
	(R_{\cS \to \cD} f)(x) = f(x), \quad \text{for all } x \in \cD.
\end{equation}
If $f \in H_q(\cS)$, then $R_{\cS \to \cD}f \in H_q(\cD)$ and $H_q(\cD) = R_{\cS \to \cD}(H_q(\cS))$ with equal norms \cite[Theorem~6]{BT04}. Thus, the pseudoinverse $R^{-1}_{\cS \to \cD}$ is an extension operator with $\norm[\cL(H_q(\cD),H_q(\cS))]{R^{-1}_{\cS \to \cD}} = 1$.

For continuous and bounded $q$, $H_q$ is separable \cite[Corollary~4]{BT04}. This yields
\begin{equation}
	\label{eq:kernel-decomp}
	q(x,\cdot) = \sum^\infty_{j=1} e_j(x)e_j(\cdot), \quad q(x,y) = \sum^\infty_{j=1} e_j(x)e_j(y), \quad x,y \in \cD,
\end{equation}
for any ONB $(e_j)_{j=1}^\infty$ with convergence in $H_q$ respectively $\R$.  When $\cD = \R^d$ and $q$ is of the form $q(x,y) = q(x-y)$, it is called stationary. If it also is continuous, bounded and integrable, it has a positive and integrable Fourier transform  $\hat{q} \colon \R^d \to \R^+$ (the spectral density) \cite[Chapter~6]{W04}. We refer to~\cite{BT04,W04} for more details on RKHSs.

\subsection{Fractional Sobolev spaces}
\label{sec:fractional-sobolev-spaces}

Let $\cD \subset \R^d$, $d=1,2,3$ be a bounded Lipschitz domain. We let $H := L^2(\cD)$ with inner product $\inpro{\cdot}{\cdot}$ and norm $\norm{\cdot}$. We let $W^{r,p}(\cD)$, $r \in [0,\infty), p \in (1,\infty),$ denote the Sobolev space equipped with the norm
\begin{equation}
	\label{eq:sobolev-slobodeckij}
	\norm[W^{r,p}(\cD)]{u} := \Big(\norm[W^{m,p}(\cD)]{u}^p + \sum_{|\alpha| = m} \int_{\cD \times \cD} \frac{|D^\alpha u(x) - D^\alpha u(y)|^p}{|x-y|^{d+p\sigma}}\dd x \dd y\Big)^{1/p},
\end{equation}
for non-integer $r = m + \sigma$, $m \in \N_0, \sigma \in (0,1)$, where $\norm[W^{m,p}(\cD)]{\cdot}$ is the integer order Sobolev norm and $|\cdot|$ the Euclidean norm on $\R^d$, see also~\cite[Definition~1.3.2.1]{G85}. The space $H^r(\cD) := W^{r,2}(\cD)$ is a Hilbert space. We write $H^r:=H^r(\cD)$ and $W^{r,p}:=W^{r,p}(\cD)$ when $\cD$ is not emphasized. The properties here and below, typically given for Sobolev spaces over $\C$, readily transfer to our setting in $\R$, cf.\ \cite[Section~2.2]{KLP22}. 

For $rp > d$, the elements of $W^{r,p}$ have continuous representatives \cite[Theorem~1.4.4.1]{G85}. From the definitions \eqref{eq:sobolev-slobodeckij} and \eqref{eq:restriction-def}, it follows directly that $R_{\cS \to \cD} \in \cL(W^{r,p}(\cS),W^{r,p}(\cD))$, a fact that we will use without reference from here on. For $p=2, r > d/2$, $\delta_x$ is bounded on $H^r$, making it a RKHS. The stationary kernel $q_r$ of the Bessel potential space $H^r(\R^d)$ is given by the continuous and bounded function $q_r(x,y) := q_r(|x-y|) := 2^{1-r} |x-y|^{r-d/2} K_{r-d/2}(|x-y|) /\Gamma(s)$ for $x, y \in \R^d$, where $K_{r-d/2}$ is the modified Bessel function of the second kind with smoothness parameter $r - d/2 > 0$ \cite[p.\ 133]{W04}. Since $\cD$ is Lipschitz, $H^r(\cD) = H_{q_r}(\cD)$ with equivalent norms, see \cite[p.\ 25]{G85} and \cite[Corollary~10.48]{W04}. The expressions (17) and (18) in \cite[Section~2.7]{S99} yield for all $r \in (d/2,d/2+1)$,
\begin{equation}
	\label{eq:matern-kernel-holder}
	q_r(x,x) - 2q_r(x,y) + q_r(y,y) = 2( q_r(0) - q_r(|x-y|)) \le C |x-y|^{2r - d}, \quad x, y \in \bar{\cD},
\end{equation}
for some constant $C < \infty$. When $r = d/2+1$, the exponent is replaced by $2r - d - \epsilon$ for arbitrary $\epsilon > 0$. When $r > d/2 +1$, it is replaced by $2$.

The space $W^{s,p_1}$ is compactly embedded into $W^{r,p_2}$ provided that $s - r > \max(d/p_1-d/p_2,0)$. The authors of \cite{ET96} quantify the decay of the entropy numbers of this embedding. Since the Sobolev space $W^{s,p}$ and Besov space $B^s_{p,p}$ coincide for non-integer $s$ (see~\cite[Section~2]{BH21}), it follows from \cite[Theorem~3.5]{ET96} that
\begin{equation}
	\label{eq:sobolev-entropy-decay}
	\norm[\cE^\rho]{I_{W^{s,p_1}(\cD) \hookrightarrow  W^{r,p_2}(\cD)}} := \norm[\cE^\rho(W^{s,p_1}(\cD),W^{r,p_2}(\cD))]{I_{W^{s,p_1}(\cD) \hookrightarrow  W^{r,p_2}(\cD)}} < \infty
\end{equation}
for Lipschitz domains $\cD$, $s, r \in \R^+\setminus\N$, $s - r > \max(d/p_1-d/p_2,0)$ and $\rho > d/(s-r)$. We use \eqref{eq:sobolev-entropy-decay} in conjunction with \eqref{eq:entropy-holder} and \eqref{eq:HS-entropy} to obtain bounds on certain Hilbert--Schmidt norms in the error analysis of Section~\ref{sec:interpolation}.

\subsection{Fractional powers of elliptic operators}

From here on, we let $\cD$ be convex and define $A$ as the part in $H$ of the operator $A \colon V \mapsto V^*$ defined by ${}_{V^*} \langle A u, v \rangle_V := a(u,v)$ for a Gelfand triple $V \hookrightarrow H \hookrightarrow V^*$ and bilinear form $a \colon V \times V \to \R$. We consider either Neumann-type homogeneous zero boundary conditions ($V := H^1$) or Dirichlet-type ($V := H^1_0(\cD) = \{v \in H^1 : \tr (v) = 0 \}$ with $\tr$ the trace operator), with
\begin{equation*}
	a(u,v) := \sum_{i,j = 1}^d \int_\cD a_{i,j} D^i u D^j v \dd x + \int_\cD c u v \dd x, \, u, v \in V.
\end{equation*}
The derivative $D^{j}$ is with respect to $x_j$, $a_{i,j} \in \cC^1(\bar{\cD})$ fulfill $a_{i,j} = a_{j,i}$ and for some $\lambda_0 > 0$ and all $y \in \R^d, x \in \cD$, $\sum^{d}_{i,j=1} a_{i,j} (x) y_i y_j \ge \lambda_0 |y|^2$. The function $c \in L^\infty(\cD)$ is non-negative a.e.\ on $\cD$. Without loss of generality, let $c \ge c_0$ for some $c_0 > 0$, a.e.\ on $\cD$, in the Neumann case. Then $a$ is continuous, symmetric and coercive, and $A$ densely defined, closed and positive definite. 

Since $V \hookrightarrow H$ is compact, $A$ has a compact inverse, enabling fractional powers by spectral decomposition. We write $\dot{H}^r := \dom(A^{r/2})$ for $r \ge 0$ and let, for $r < 0$, $\dot{H}^r$ be the completion of $H$ under $\norm{A^{r/2} \cdot}$. Since $\cD$ is convex, the functions of $\dot{H}^2$ are the elements of $H^2$ that satisfy the boundary conditions of $A$ while
\begin{equation}
	\label{eq:sobolev_id_theta_def}
	\dot{H}^r = H^r, \, r \in [0,\theta) \text{ with } \theta := \begin{cases}
		1/2 & \text{ for Dirichlet boundary conditions and} \\
		3/2 & \text{ for Neumann boundary conditions,}
	\end{cases}
\end{equation} 
with norm equivalence \cite[Theorems~16.9, 16.13]{Y10}. Moreover, $\dot{H}^r \hookrightarrow H^r, r \in [0,2],$ by complex interpolation \cite[Theorem~3.2, Remark~3.3]{BH21}. We refer to \cite[Appendix~B]{K14} and \cite[Chapter~2]{Y10} for more on this setting.

\section{Parabolic SPDEs driven by cylindrical Wiener processes in RKHSs}
\label{sec:regularity}

Here, mild solutions to~\eqref{eq:spde} and cylindrical Wiener processes $W$ in $H_q(\cD)$ are introduced. We extend $W$ to $\cS$, state our assumptions on $q$ and derive regularity in the setting of stochastic reaction-diffusion-advection equations.

\subsection{Mild solutions of SPDEs}

For $T < \infty$, let $(\Omega,\cF,(\cF_t)_{t \in [0,T]},\IP)$ be a filtered probability space under the usual conditions. The process on $[0,T]$ given by 
\begin{equation}
	\label{eq:mild-solution}
	X(t) := {S}(t) X(0) + \int_{0}^{t} {S}(t-s) F (X(s)) \dd s + \int_0^t {S}(t-s) G(X(s)) \dd W(s),
\end{equation}
is called mild if also $X \in \cC([0,T],L^p(\Omega,H))$, $p \ge 2$, where $S = (S(t))_{t \ge 0}$ is the analytic semigroup generated by $-A$. The stochastic integral is of It\^o type \cite[Chapter~4]{DPZ14} and $W$ a standard cylindrical Wiener process in $H_q(\cD)$ (see Section~\ref{sec:cyl-wiener-rkhs} below). In~\cite{K14}, the existence of a unique solution~\eqref{eq:mild-solution} is derived with the regularity estimate $\sup_{t \in [0,T]} \norm[L^p(\Omega,\dot{H}^{r+1})]{X(t)} < \infty.$ This is shown in \cite[Theorem~2.27]{K14} under the following assumption.
\begin{assumption}
	\label{assumption:regularity}
	For some $r \in [0,1)$, $p \ge 2$, there is a $C < \infty$ such that:
	\begin{enumerate}[label=(\alph*)]
		\item \label{assumption:regularity:X0} The initial value $X_0$ is $\cF_0$-measurable and $\norm[L^p(\Omega,\dot{H}^{r+1})]{X_0} \le C$.
		\item \label{assumption:regularity:F} The non-linear reaction-advection term $F \colon H \to \dot{H}^{-1}$ satisfies
		\begin{enumerate}[label=(\roman*)]
			\item $\norm[\dot{H}^{-1}]{F(u) - F(v)} \le C  \norm[H]{u - v}$ for all $u, v \in H$ 
			\item $\norm[\dot{H}^{r-1}]{F(u) - F(v)} \le C  \norm[\dot{H}^{r}]{u - v}$ for all $u, v \in \dot{H}^r$.
		\end{enumerate}		
		\item \label{assumption:regularity:G} The non-linear multiplicative noise operator $G \colon H \to \cL_2(H_q,H)$ satisfies
		\begin{enumerate}[label=(\roman*)]
			\item $\norm[\cL_2(H_q,\dot{H}^r)]{G(u)} \le C ( 1 + \norm[\dot{H}^r]{u})$ for all $u \in \dot{H}^r$ and
			\item $\norm[\cL_2(H_q,H)]{G(u)-G(v)} \le C \norm{u - v}$ for all $u, v \in H$.
		\end{enumerate}
	\end{enumerate}
\end{assumption}

Assumption~\ref{assumption:regularity}\ref{assumption:regularity:X0} depends on Sobolev regularity of $X_0$ and whether it satisfies the boundary conditions associated with $A$. Assumption~\ref{assumption:regularity}\ref{assumption:regularity:F} is with
\begin{equation}
	\label{eq:F1-F2-def}
	F(u) := F_1(u)+ F_2(u) := b(\cdot) \cdot \nabla u + f(u(\cdot),\cdot)
\end{equation}
satisfied for regular $b = (b_1,\ldots,b_d) \colon \cD \to \R^d$, $f \colon \R \times \cD \to \R$ and functions $u$ on $\cD$. The next proposition gives sufficient conditions on $f$ and $b$ for this to hold true.

\begin{proposition}
	\label{prop:F-reg}
	Let $F$ be given by~\eqref{eq:F1-F2-def}, where for some constant $C < \infty$,
	\begin{enumerate}[label=(\roman*)]
		\item $\int_{\cD} |f(0,x)|^2 \dd x \le C$
		\item $\esssup_{y \in \cD} |f(x_1,y) - f(x_2,y)| \le C |x_1-x_2|$, $x_1, x_2 \in \R$, and
		\item $\esssup_{x \in \cD}  \sum_{j=1}^2 \left(|b_j(x)| + \left|\frac{\partial b_j}{\partial x_j}(x) \right| \right) \le C$.
	\end{enumerate}
	Then $F$ extends to an operator on $H$ and fulfills Assumption~\ref{assumption:regularity}\ref{assumption:regularity:F} for all $r \in [0,1]$
	\begin{enumerate}[label=(\roman*)]
		\item under Dirichlet boundary conditions on $A$ and
		\item under Neumann boundary conditions if also $b(x) = 0$ for a.e.\ $x \in \partial \cD$.
	\end{enumerate}
\end{proposition}
This result is derived for $F = F_2$ in \cite[p.\ 120]{JR12}, recalling that $\dot{H}^r \hookrightarrow H$ for $r \ge 0$. The term $F_1$ is treated in \cite[Example~2.22]{K14} in the Dirichlet case. The Neumann case follows analogously to this from Green's formula, since $b$ vanishes on $\partial \cD$.

For which $r$ Assumption~\ref{assumption:regularity}\ref{assumption:regularity:G} is satisfied when $G$ is a multiplication operator
\begin{equation}
	\label{eq:G-def}
	(G(u)v)(x) := g(u(x),x)v(x), \quad u \in H, v \in H_q, x \in \cD,
\end{equation}
depends largely on the covariance kernel $q$. In the next section, we specify our assumptions on $q$ and formally introduce the mathematical model of the noise $W$.

\subsection{Cylindrical Wiener processes in RKHSs}
\label{sec:cyl-wiener-rkhs}

A family $W$, indexed by $t \in [0,T]$, of linear maps from a Hilbert space $U$ to the space of real-valued random variables on $(\Omega,\cF,(\cF_t)_{t \in [0,T]},\IP)$ is called a cylindrical stochastic process in $U$ adapted to $(\cF_t)_{t \in [0,T]}$ if $W(t) u$ is $\cF_t$-measurable for all $t \ge 0$ and all $u \in U$. Such a process $W$ is a (strongly) cylindrical Wiener process with covariance operator $K \in \Sigma^+(U)$ if $t \mapsto (W(t)u_1,\ldots,W(t)u_n)$ is a Wiener process in $\R^n$ for $u_1, \ldots, u_n \in U, n \in \N,$ and $\E[(W(t)u)(W(t)v)] = t \inpro[U]{Ku}{v}$ for all $t \in [0,T], u,v \in U$. It is called standard if $K = I$. For such $W$ in $U$, another Hilbert space $\tilde U$ and $\Gamma \in \cL(U,\tilde U)$ the process $\Gamma W$ defined by $t \mapsto W(t)(\Gamma^* v)$, $v \in \tilde U$, is a cylindrical Wiener process in $\tilde U$ with covariance $\Gamma K \Gamma^* \in \Sigma^+(\tilde U)$. Here $\Gamma^* \in \cL(\tilde U,U)$ is the adjoint of $\Gamma$. These claims, along with the fact that $W$ exists for a given $U$ and $K$ are shown in \cite{R11}. We have followed the setting of this paper and \cite[Section 4.1.2]{DPZ14}, identifying Hilbert spaces with their duals.

We now specify $W$ to be a standard ($K=I$) cylindrical Wiener process in $H_q(\cD)$. With $W(t,x) := \delta_x W(t)$, one obtains by the reproducing property of $H_q$ a Gaussian random field with covariance $\E[W(t,x)W(s,y)] = \min(t,s)q(x,y), s,t \in [0,T], x,y \in \cD$. This yields that the interpolation \eqref{eq:IhW-expansion}, used for simulation of $W$ by sampling on $\cS \supseteq \cD$, is well-defined. However, we also need to confirm that an extension $\tilde W$ of $W$ to a standard cylindrical Wiener process in $H_q(\cS)$ exists, with $W = R_{\cS \to \cD} \tilde W$.

\begin{lemma}
	\label{lem:extension-of-W}
	If $q$ is positive semidefinite on $\cS$ there is a standard cylindrical Wiener process $\tilde W$ on $H_q(\cS)$ such that $(R_{\cS \to \cD} \tilde W) u  = W u$ in $L^2(\Omega,\R)$ for all $u \in U$. 
\end{lemma}

\begin{proof}
	Recall that $\ker(R_{\cS \to \cD})$ is a Hilbert subspace of  $H_q(\cS) = \ker(R_{\cS \to \cD}) \oplus \ker(R_{\cS \to \cD})^\perp$. Let $\hat{W}$ be a cylindrical Wiener process on $H_q(\cS)$, independent of $W$, with covariance operator $P_{\ker(R_{\cS \to \cD})}$, the orthogonal projection onto $\ker(R_{\cS \to \cD})$. Set $\tilde W := R_{\cS \to \cD}^{-1} W + \hat{W}$. Since $H_q(\cD) = R_{\cS \to \cD}(H_q(\cS))$, $R_{\cS \to \cD}^{-1}$ is a bounded operator on $H_q(\cD)$. This means that $R_{\cS \to \cD}^{-1} W$ is a cylindrical Wiener process in $\ker(R_{\cS \to \cD})^\perp \subset H_q(\cS)$. By independence, $t \mapsto (\tilde W(t)u_1,\ldots,\tilde W(t)u_n)$ is an $\R^n$-valued Wiener process for all $u_1, \ldots, u_n \in H_q(\cS), n \in \N$ and for all $u,v \in H_q(\cS)$, $\E[(\tilde W (t)u)(\tilde W (t)v)]$ is given by
	\begin{equation}
		\label{eq:lem:extension-of-W:1}
		\begin{split}
			&\E[(W (t) (R_{\cS \to \cD}^{-1})^* u)(W (t)(R_{\cS \to \cD}^{-1})^* v)] + \E[(\hat W (t)u)(\hat W (t)v)] \\
			&\quad=t\inpro[H_q(\cS)]{R_{\cS \to \cD}^{-1} (R_{\cS \to \cD}^{-1})^*u}{v} + t\inpro[H_q(\cS)]{P_{\ker(R_{\cS \to \cD})} u}{v}.
		\end{split}
	\end{equation}
	Since $v \in \ker(R_{\cS \to \cD})^\perp$ yields $v = R_{\cS \to \cD}^{-1} R_{\cS \to \cD} v$, and since $H_q(\cD) = R_{\cS \to \cD}(H_q(\cS))$,
	\begin{equation}
		\label{eq:lem:extension-of-W:2}
		\begin{split}
			\inpro[H_q(\cS)]{R_{\cS \to \cD}^{-1} (R_{\cS \to \cD}^{-1})^*u}{v} &= \inpro[H_q(\cS)]{R_{\cS \to \cD}^{-1} (R_{\cS \to \cD}^{-1})^*u}{R_{\cS \to \cD}^{-1} R_{\cS \to \cD}v} \\ 
			&= \inpro[H_q(\cD)]{(R_{\cS \to \cD}^{-1})^*u}{R_{\cS \to \cD}v} \\
			&= \inpro[H_q(\cS)]{u}{P_{\ker(R_{\cS \to \cD})^\perp}v}, \quad v \in \ker(R_{\cS \to \cD})^\perp.
		\end{split}
	\end{equation}
	Note now that for $x \in \cD$ and $v \in H_q(\cS)$, $q(x,\cdot) = P_{\ker(R_{\cS \to \cD})^\perp} q(x,\cdot)$ as a result of \begin{equation*}\inpro[H_q(\cS)]{(I-P_{\ker(R_{\cS \to \cD})^\perp})q(x,\cdot)}{v} = \inpro[H_q(\cS)]{q(x,\cdot)}{P_{\ker(R_{\cS \to \cD})} v} = (P_{\ker(R_{\cS \to \cD})} v)(x).
	\end{equation*} 
	Since also $P_{\ker(R_{\cS \to \cD})^\perp} = R_{\cS \to \cD}^{-1} R_{\cS \to \cD}$ we find that $q(x,\cdot) = R_{\cS \to \cD}^{-1} q(x,\cdot)|_{\cD}$ for $x \in \cD$. Combining this with the reproducing property, we find that for $x \in \cD$, $v \in \ker(R_{\cS \to \cD})$, $\inpro[H_q(\cS)]{R_{\cS \to \cD}^{-1} q(x,\cdot)|_{\cD}}{v} = \inpro[H_q(\cS)]{q(x,\cdot)}{v}$. By approximating $(R_{\cS \to \cD}^{-1})^*u$ with linear combinations of $(q(x,\cdot))_{x \in \cD}$ in $H_q(\cD)$ (see \cite[Theorem~3]{BT04}), we then obtain
	\begin{equation}
		\label{eq:lem:extension-of-W:3}
		\inpro[H_q(\cS)]{R_{\cS \to \cD}^{-1} (R_{\cS \to \cD}^{-1})^*u}{v} = 0 = \inpro[H_q(\cS)]{u}{P_{\ker(R_{\cS \to \cD})^\perp}v}, \quad v \in \ker(R_{\cS \to \cD}).
	\end{equation}
	Combining~\eqref{eq:lem:extension-of-W:2} and~\eqref{eq:lem:extension-of-W:3} in~\eqref{eq:lem:extension-of-W:1} with the identity $P_{\ker(R_{\cS \to \cD})^\perp} + P_{\ker(R_{\cS \to \cD})} = I$ yields $\E[(\tilde W (t)u)(\tilde W (t)v)] = t\inpro[H_q(\cS)]{u}{v}$ which shows that $\tilde W$ is a standard cylindrical Wiener process on $H_q(\cS)$. The fact that $R_{\cS \to \cD} \tilde W u = W u$ in $L^2(\Omega,\R)$ now follows from the identities $R_{\cS \to \cD} P_{\ker(R_{\cS \to \cD})} = 0$ and $R_{\cS \to \cD} R^{-1}_{\cS \to \cD} = I$ on $R_{\cS \to \cD}(H_q(\cS)) = H_q(\cD)$. 
\end{proof}

We now state our main assumption for $q$ on $\cS \supseteq \cD$. Note that if $q$ satisfies this on $\cS$, the corresponding assumption is also fulfilled for $\cD$. The first part is used to verify Assumption~\ref{assumption:regularity}\ref{assumption:regularity:G}, the second for the numerical analysis of Section~\ref{sec:interpolation}.

\begin{assumption}\label{assumption:q}
	The positive semidefinite kernel $q$ is bounded on $\cS$ and satisfies
	\begin{enumerate}[label=(\alph*)]
		\item \label{assumption:q-holder} for some $\gamma \in (0,1]$ and $C < \infty$ the H\"older-type regularity condition
		\begin{equation*}
			q(x,x) - 2q(x,y) + q(y,y) \le C |x-y|^{2\gamma}, \quad x,y \in \cS,
		\end{equation*}
		\item \label{assumption:q-sobolev} and for some $\mu > 0$ and $\psi > d/\mu$ the embedding property $H_q(\cS) \hookrightarrow W^{\mu, \psi}(\cS)$.
	\end{enumerate}
\end{assumption}

\begin{remark}
	\label{rem:Q-Wiener}
	Since $q$ is bounded, $H_q(\cD) \hookrightarrow H = L^2(\cD)$. The embedding is Hilbert--Schmidt, since continuity of $q$ yields an ONB $(e_j)_{j=1}^\infty$ of $H_q$ such that
	\begin{equation*}
		\norm[\cL_2]{I_{H_q \hookrightarrow H}}^2 := \norm[\cL_2(H_q,H)]{I_{H_q \hookrightarrow H}}^2 = \sum_{j=1}^\infty \int_\cD e_j(x)^2 \dd x = \int_\cD q(x,x) \dd x < \infty,
	\end{equation*}
	see~\eqref{eq:kernel-decomp}. Therefore, $W$, viewed as cylindrical in $H$, is under Assumption~\ref{assumption:q} induced by a so called $Q$-Wiener process $W_Q$ with covariance $Q := I_{H_q \hookrightarrow H} I^*_{H_q \hookrightarrow H} \in \Sigma^+(H)$ \cite[Theorem~8.1]{R11}. The relation to the original process is given by $W(t)(I^*_{H_q \hookrightarrow H} u) = \inpro{W_Q(t)}{u}$ for $u \in H$. The operator $Q$ is of integral type with kernel $q$. We could equivalently work with $W_Q$ instead of $W$ in \eqref{eq:spde}, but the frequent use of the RKHS properties in Section~\ref{sec:rkhs} would make the notation and analysis cumbersome.
\end{remark}

\subsection{Non-linear multiplicative noise operator}

We prove a result relating the H\"older-type regularity of $q$ in Assumption~\ref{assumption:q}\ref{assumption:q-holder} to Assumption~\ref{assumption:regularity}\ref{assumption:regularity:G} on $G$ when $g \colon \R \times \cD \to \R$ in~\eqref{eq:G-def} satisfies, for some constant $C < \infty$,
\begin{equation}
	\label{eq:g-lipschitz}
	|g(x_1,y_1) - g(x_2,y_2)| \le C (|x_1-x_2| + |y_1-y_2|), \quad x_1,x_2 \in \R, y_1,y_2 \in \cD.
\end{equation}
In the special case where $q$ has a known expansion (Example~\ref{ex:kernel-basis} below), $A$ has Dirichlet boundary conditions and $\cD= [0,1]^d, d \in \N$, it was derived in \cite[Section~4]{JR12}. We extend this setting to more general kernels, domains and boundary conditions.

\begin{proposition}
	\label{prop:G-reg}
	Let $G$ be given by~\eqref{eq:G-def} where $g$ satisfies~\eqref{eq:g-lipschitz}. Under Assumption~\ref{assumption:q}\ref{assumption:q-holder}, there is for all $r \in [0,\gamma)$ a constant $C < \infty$ such that
	\begin{enumerate}[label=(\roman*)]
		\item \label{eq:prop:G-reg:lipschitz} $\norm[\cL_2(H_q(\cD),H)]{G(u)-G(v)} \le C \norm{u - v}$ for all $u, v \in H$ and
		\item \label{eq:prop:G-reg:bound}$\norm[\cL_2(H_q(\cD),H^r)]{G(u)} \le C ( 1 + \norm[H^r]{u})$ for all $u \in H^r$.
	\end{enumerate}
	Therefore, Assumption~\ref{assumption:regularity}\ref{assumption:regularity:G} is satisfied for all $r < \min(\gamma,\theta)$, with $\theta$ as in~\eqref{eq:sobolev_id_theta_def}. 
\end{proposition}

\begin{proof}
	Note first that Assumption~\ref{assumption:q}\ref{assumption:q-holder} along with the reproducing property of $H_q$ yield that all $v \in H_q$ are bounded. Moreover, \eqref{eq:g-lipschitz} ensures that $\norm{g(u(\cdot),\cdot)}^2 \lesssim 1 + \norm{u}^2$ for $u \in H$ so that $G(u)(v)$ indeed takes values in $H$ for $v \in H_q$.
	
	By Assumption~\ref{assumption:q}\ref{assumption:q-holder}, $H_q$ is separable with an ONB $(e_j)_{j=1}^\infty$. By Tonelli's theorem and the decomposition~\eqref{eq:kernel-decomp} for the bounded kernel $q$, $G$ maps into $\cL_2(H_q(\cD),H)$ since
	\begin{align*}
		\norm[\cL_2(H_q,H)]{G(u)}^2 &= \sum_{j=1}^\infty \int_{\cD} g(u(x),x)^2  e_j(x)^2 \dd x  \\ &\le \left(\sup_{x \in \cD} q(x,x)\right) \int_{\cD} g(u(x),x)^2 \dd x \lesssim  \int_{\cD} x^2 + |u(x)|^2 \dd x.
	\end{align*}
	Similarly, the claim~\ref{eq:prop:G-reg:lipschitz} is obtained from the Lipschitz regularity of $g$, as
	\begin{equation*}
		\norm[\cL_2(H_q,H)]{G(u)-G(v)}^2 \le \left(\sup_{x \in \cD} q(x,x)\right) \int_{\cD} \left( g(u(x),x) - g(v(x),x) \right)^2  q(x,x) \dd x.
	\end{equation*}
	can be bounded by $C\norm{u(x)-v(x)}^2$ for some $C<\infty$. To prove~\ref{eq:prop:G-reg:bound}, we note that the definition of $\norm[H^r]{\cdot}=\norm[W^{r,2}]{\cdot}$ yields that $\norm[\cL_2(H_q,H^r)]{G(u)}^2$ is given by
	\begin{align*}
		\sum_{j=1}^\infty \int_{\cD} |(G(u)e_j)(x)|^2 \dd x + \sum_{j=1}^\infty \int_{\cD \times \cD} \frac{|(G(u)e_j)(x) -  (G(u)e_j)(y)|^2}{|x-y|^{d+2r}}\dd x \dd y.
	\end{align*}
	The first of these terms, $\norm[\cL_2(H_q,H)]{G(u)}^2$, was bounded above. For the second,
	\begin{equation}
		\begin{split}
			\label{eq:prop:G-reg:2}
			&\sum_{j=1}^\infty \int_{\cD \times \cD} \frac{|(G(u)e_j)(x) -  (G(u)e_j)(y)|^2}{|x-y|^{d+2r}}\dd x \dd y \\ 
			&\quad= \int_{\cD \times \cD} \frac{ \sum_{j=1}^\infty |g(u(x),x)e_j(x) - g(u(y),y)e_j(y)|^2}{|x-y|^{d+2r}}\dd x \dd y \\ 
			&\quad\le 2 \int_{\cD \times \cD} \frac{ \big(\sum_{j=1}^\infty 
				(e_j(y))^2\big) |g(u(x),x) - g(u(y),y)|^2 }{|x-y|^{d+2r}}\dd x \dd y \\
			&\qquad+ 2 \int_{\cD \times \cD} \frac{\big(\sum_{j=1}^\infty (e_j(x) - e_j(y))^2\big)|g(u(x),x)|^2}{|x-y|^{d+2r}}\dd x \dd y.
		\end{split}
	\end{equation}
	In the first integral, we use \eqref{eq:kernel-decomp}, \eqref{eq:g-lipschitz} and  Assumption~\ref{assumption:q}\ref{assumption:q-holder} to see that
	\begin{align*}
		&\int_{\cD \times \cD} q(y,y) \frac{|g(u(x),x) - g(u(y),y)|^2}{|x-y|^{d+2r}}\dd x \dd y \\ &\quad\le \left(\sup_{x \in \cD} q(x,x)\right) \int_{\cD \times \cD} \frac{|g(u(x),x) - g(u(y),y)|^2}{|x-y|^{d+2r}}\dd x \dd y \\
		&\quad\lesssim \int_{\cD \times \cD} \frac{|u(x) - u(y)|^2}{|x-y|^{d+2r}}\dd x \dd y + \int_{\cD \times \cD} |x-y|^{2(1-r) - d} \dd x \dd y \lesssim \norm[H^r]{u}^2 + 1,
	\end{align*}
	since $r < \gamma \le 1$. The second integral in~\eqref{eq:prop:G-reg:2} is bounded by
	\begin{align*}
		&\int_{\cD \times \cD} \frac{g(u(x),x)^2 \big(q(x,x)-2q(x,y)+q(y,y)\big)}{|x-y|^{d+2r}}\dd x \dd y \\
		&\quad\le \Big(\sup_{x \neq y \in \cD} \frac{q(x,x)-2q(x,y)+q(y,y)}{|x-y|^{2\gamma}}\Big) \\
		&\hspace{3em}\times \left(\sup_{x \in \cD} \int_{\cD} |x-y|^{2(\gamma-r)-d} \dd y\right) \int_{\cD} |g(u(x),x)|^2 \dd x \lesssim 1 + \norm{u}^2.
	\end{align*}
	
	Assumption~\ref{assumption:regularity}\ref{assumption:regularity:G} follows directly from the fact that $\dot{H}^r = H^r$ for $r < \theta$ in~\eqref{eq:sobolev_id_theta_def}.
\end{proof}

Since $\gamma \le 1$, the bound by $\theta$ on $r$ in Proposition~\ref{prop:G-reg} only applies in the Dirichlet case. It can be dropped if $q(x,\cdot)= 0$ or if $g(0,x) = 0$ for all $x \in \partial \cD$ \cite[Theorem~16.13]{Y10}. Moreover, the result is sharp. Too see this, consider Neumann boundary conditions and $H_q(\cD) := H^{s}$ for $s \in (d/2,d/2+1)$, so that $\gamma = s - d/2$ in Assumption~\ref{assumption:q}\ref{assumption:q-holder} by~\eqref{eq:matern-kernel-holder}. If we let $g(x,y) := 1$ for $x \in \R, y \in \cD$, we have $\norm[\cL_2(H_q,H^r)]{G(u)} = \norm[\cL_2]{I_{H^s \hookrightarrow H^r}}$ for $u \in H^r$. Proposition~\ref{prop:G-reg} shows that $\norm[\cL_2]{I_{H^s \hookrightarrow H^r}} < \infty$ for $r < s - d/2$ whereas $\norm[\cL_2]{I_{H^s \hookrightarrow H^r}} = \infty$ if $r = s - d/2$, see, e.g., \cite[Lemma~2.3]{KLP22}.

\section{Piecewise linear interpolation of noise in SPDE approximations}
\label{sec:interpolation}

By Lemma~\ref{lem:extension-of-W} we write without loss of regularity the mild solution $X$ of~\eqref{eq:mild-solution} as
\begin{equation*}
	X(t) = {S}(t) X(0) + \int_{0}^{t} {S}(t-s) F (X(s)) \dd s + \int_0^t {S}(t-s) G(X(s)) R_{\cS \to \cD} \dd \tilde W(s),
\end{equation*}
for $t \in (0,T]$ and a standard cylindrical Wiener process $\tilde W$ on $H_q(\cS)$. We approximate $X$ by $X_{h,\Delta t}$ through the semi-implicit Euler scheme and piecewise linear finite elements. This was analyzed in~\cite{K14}, here we consider the additional noise discretization error coming from sampling $\tilde W$ pointwise on $\cS$ and interpolating onto $\cD$.

\subsection{Discretization setting}

From here on, we let $\cD$ and $\cS$ be polyhedral, as is common for finite element methods to avoid dealing with curved boundaries \cite{BS08}. Consider two triangulations $\cT_h,\cT_{h'}$ on $\cD$ and $\cS$ with maximal mesh sizes $h, h' \in (0,1]$, and let $V_h \subset H^1(\cD)$ and $V_{h'} \subset H^1(\cS)$ be the spaces of piecewise linear functions on these. By $I_h$ and $I_{h'}$ we denote the piecewise linear interpolants with respect to $\cT_h$ and $\cT_{h'}$. For a function $v$ on $\cD$, $v_h =: I_h v$ is the unique function $v_h \in V_h$ such that $v_h(x_j) = v(x_j)$ for all nodes $x_j$ of $\cT_h$. In terms of the nodal basis functions $\varphi^h_j$,
\begin{equation*}
	I_h v = \Big(\sum_{j=1}^{N_h} \delta_{x_j} \varphi^h_j\Big) v = \sum_{j=1}^{N_h} v(x_j) \varphi^h_j,
\end{equation*}
where $N_h = \dim(V_h)$ is the number of nodes. We need two results on these interpolants. The first generalizes the results in \cite{BB01} to $p \neq 2$, with a proof in Appendix~\ref{sec:appendix}.

\begin{proposition}
	\label{prop:interpolation}
	Let $\cD \subset \R^d, d \le 3$, be polyhedral with regular triangulation $\cT_h$. Let $1 < p < \infty$, $sp > d$ and $r \in [0,s] \cap [0,1+1/p)$. Assume also that $\cT_h$ is quasiuniform for $s < 1$ and $r > 0$. Then, there is a $C < \infty$ such that
	\begin{equation*}
		\norm[W^{r,p}]{(I - I_h)v} \le C h^{\min(s-r,2)} \norm[W^{s,p}]{v} \text{ for all } h \in (0,1], v \in W^{s,p}(\cD).
	\end{equation*}
\end{proposition}
We recall the notation $W^{s,p}$ for $W^{s,p}(\cD)$ when the spatial domain in the fractional Sobolev space is clear from context, see Section~\ref{sec:fractional-sobolev-spaces}.
\begin{lemma}
	\label{lem:hs-interpolant-bound}
	Let $\cD \subset \cS \subset \R^d, d \le 3,$ be polyhedral domains with triangulations $\cT_h$ and $\cT_{h'}$. Let $q$ be a bounded and continuous positive semidefinite kernel on $\cS$ and let $B \in \cL(L^\infty(\cD),L^2(\cD))$ be given by $(B v)(x) := b(x)v(x)$ for a.e.\ $x \in \cD$. Then
	\begin{equation*}
		\norm[\cL_2(H_q(\cS),L^2(\cD))]{B I_h R_{\cS \to \cD} I_{h'}} \le 2 \sup_{x,y \in \cS} \sqrt{|q(x,y)|} \norm[L^2(\cD)]{b} \text{ for all } b \in L^2(\cD). 
	\end{equation*}
\end{lemma}
\begin{proof}
	Let us write $\varphi^h$ and $\varphi^{h'}$ for the nodal bases on $V_h$ and $V_{h'}$. Then 
	\begin{equation*}
		\norm[\cL_2(H_q(\cS),L^2(\cD))]{B I_h R_{\cS \to \cD} I_{h'}}^2 = \sum_{j = 1}^\infty \int_{\cD} \Big| b(x) \sum_{k=1}^{N_{h'}} \sum_{\ell=1}^{N_h} \varphi^{h'}_k(x_\ell) \varphi^h_\ell(x) e_j(x_k) \Big|^2	\dd x.
	\end{equation*}
	Using~\eqref{eq:kernel-decomp}, this can be written as 
	\begin{align*}
		&\int_{\cD} b(x)^2 \sum_{k,m=1}^{N_{h'}} \sum_{\ell,n=1}^{N_h} \varphi^{h'}_k(x_\ell) \varphi^{h'}_m(x_n) \varphi^h_\ell(x) \varphi^h_\ell(x) \varphi^h_n(x) q(x_k,x_m) \dd x \\
		&\quad= \sum_{m=1}^{N_{h'}} \sum_{n=1}^{N_h} \int_{\cD} b(x)^2 \left(I_h R_{\cS \to \cD} I_{h'} q(\cdot,x_m)\right)(x) \varphi^h_n(x) \varphi^{h'}_m(x_n) \dd x.
	\end{align*}
	By stability of $I_h$ and $I_{h'}$ in $L^\infty$, $|I_h R_{\cS \to \cD} I_{h'} q(\cdot,x_m)| \le \sup_{x,y \in \cS} |q(x,y)|$. Moreover, $\varphi^h, \varphi^{h'}$ are non-negative, sum to $1$ and $\varphi^{h'}_m(x_n) \neq 0$ for at most $4$ indices $m$. Thus, 
	\begin{align*}
		\norm[\cL_2(H_q(\cS),L^2(\cD))]{B I_h R_{\cS \to \cD} I_{h'}}^2 &\le 4 \sup_{x,y \in \cS} |q(x,y)| \int_{\cD} b(x)^2 \sum_{n=1}^{N_h} \varphi^h_n(x) \dd x \\ &= 4 \sup_{x,y \in \cS} |q(x,y)| \int_{\cD} b(x)^2 \dd x \\ &= 4 \sup_{x,y \in \cS} |q(x,y)| \norm[L^2(\cD)]{b}^2.
	\end{align*}
\end{proof}

We next set $\dot{V}_h := V_h \cap \dot{H}^1$ and denote by $\dot P_h \colon H \to \dot V_h$ and $\dot R_h \colon \dot{H}^1 \to \dot V_h$ the orthogonal projections under the topologies of $H$ and $\dot{H}^1$, respectively. Let $A_h \in \Sigma^+(V_h)$ be defined by $\inpro{A_h u_h}{v_h} := a(u_h,v_h)$ for $u_h,v_h \in \dot V_h$ so that $\norm{A_h^{1/2} v_h} = \norm[\dot{H}^1]{v_h}$. Combining this equality with \cite[(3.15)]{K14}, we find $C_1, C_2 \in (0,\infty)$ such that
\begin{equation}
	\label{eq:disc-norm-equiv}
	C_1 \norm[\dot{H}^r]{v_h} \le \norm[H]{A_h^{\frac{r}{2}} v_h} \le C_2 \norm[\dot{H}^r]{v_h}, \quad v_h \in \dot{V}_h, h \in (0,1], r \in [-1,1],
\end{equation}
using also an interpolation argument. With this, we can introduce $X_{h,\Delta t}$. For a time step ${\Delta t}\in(0,1]$, let $(t_j)_{j =  0}^\infty$ be given by $t_j={\Delta t} j$ and $N_{\Delta t} +1 = \inf\{j\in\N : t_j \notin [0,T]\}$. Set $X_{h,\Delta t}^0 := \dot{P}_h X_0$ and let $X_{h,\Delta t}^j$ solve
\begin{equation}
	\label{eq:spde-approximation}
	X_{h,\Delta t}^j - X_{h,\Delta t}^{j-1} + \Delta t (A_hX_{h,\Delta t}^j + \dot{P}_h F(X_{h,\Delta t}^{j-1}) ) = \dot{P}_h G(X_{h,\Delta t}^{j-1}) I_{h} R_{\cS \to \cD} I_{h'} \Delta \tilde W^j
\end{equation}
for $j = 1, \ldots, N_{\Delta t}$. Here $\Delta \tilde W^j := \tilde W(t_j)- \tilde W(t_{j-1})$. This is the same approximation as that of~\cite[Section~3.6]{K14}, but with $\Delta W^j := R_{\cS \to \cD} \Delta \tilde W^j$ replaced by $I_{h} R_{\cS \to \cD} I_{h'} \Delta \tilde W^j$.

Let us write $S_{h,{\Delta t}}:=(I +{\Delta t} A_h)^{-1} \dot{P}_h$. By \cite[Lemma~7.3]{T06}, for some $C < \infty$,
\begin{equation}
	\label{eq:disc-semi-smooth}
	\norm[\cL(H)]{A_h^\frac{r}{2} S_{h,{\Delta t}}^j } \le C t_j^{-r/2}, \quad h \in (0,1], r \in [0,2], j  = 1, \ldots, N_{\Delta t}+1,
\end{equation}
and by \cite[Lemma~3.12]{K14} (building on \cite[Theorem~7.8]{T06}), there is a $C < \infty$ such that
\begin{equation}
	\label{eq:disc-semi-error}
	\norm[\cL(H)]{(S(t_j) - S^j_{h,\Delta t})A^{-\frac{r}{2}}} \le C (h^s + \Delta t^{\frac{s}{2}}) t_j^{-\frac{s-r}{2}}, \quad h \in (0,1], j  = 1, \ldots, N_{\Delta t}+1,
\end{equation} 
where $0 \le r \le s \le 2$. Estimate~\eqref{eq:disc-semi-error} requires a bound on $\dot{R}_h$ (see \cite[Assumption~3.3]{K14}) to be fulfilled. This follows from the fact that since $\cD$ is convex, elliptic regularity yields $\sup_{h \in (0,1]} h^{-r} \norm[\cL(\dot{H}^r,H)]{I-\dot{R}_h} \lesssim 1$ for $r \in \{1,2\}$ (see \cite[p. 799]{FS91}). 

For all $n = 0, \ldots, N_{\Delta t}$, we can write $X^n_{h,\Delta t}$ in closed form as
\begin{equation*}
	X^n_{h,\Delta t} = S^n_{h,\Delta t}  X_0 + \Delta t \sum^{n-1}_{j=0} S^{n-j}_{h,\Delta t}  F(X^{j}_{h,\Delta t}) + \sum^{n-1}_{j=0} S^{n-j}_{h,\Delta t}  G(X_{h,\Delta t}^{j}) I_{h} R_{\cS \to \cD} I_{h'} \Delta \tilde W^j.
\end{equation*} 

\subsection{The main result}

In this section, we derive an error bound on $X - X_{h,\Delta t}$ under noise interpolation. First, we need a regularity estimate for $X_{h,\Delta t}$.
\begin{lemma}
	\label{lem:disc-regularity}
	Suppose $X_0$ and $F$ satisfy Assumption~\ref{assumption:regularity}\ref{assumption:regularity:X0}-\ref{assumption:regularity:F} for some $p \ge 2, r \in [0,1)$. Suppose also that $q$ satisfies Assumption~\ref{assumption:q}\ref{assumption:q-holder} and let $G$ in~\eqref{eq:G-def} satisfy~\eqref{eq:g-lipschitz}. Then, there is for all $s \in [0,1)$ a constant $C < \infty$ such that 
	\begin{equation*}
		\sup_{n \in \{1, \ldots, N_{\Delta t}\}} \norm[L^p(\Omega,\dot{H}^s)]{X_{h,\Delta t}^n} \le C \text{ for all } \Delta t, h, h' \in (0,1].
	\end{equation*}
\end{lemma}

\begin{proof}
	By the Burkholder--Davis--Gundy (BDG) inequality \cite[Theorem~4.37]{DPZ14}:
	\begin{equation}
		\label{eq:lem:disc-regularity:pfeq1}
		\begin{split}
			\norm[L^p(\Omega,\dot{H}^s)]{X_{h,\Delta t}^n}^2 &\lesssim \norm[L^p(\Omega,\dot{H}^s)]{S^n_{h,\Delta t}  X_0}^2 + \Bignorm[L^p(\Omega,\dot{H}^s)]{\Delta t \sum^{n-1}_{j=0} S^{n-j}_{h,\Delta t}  F(X^{j}_{h,\Delta t})}^2 \\ &\quad+ \Delta t \sum^{n-1}_{j=0}  \norm[L^p(\Omega,\cL_2(H_q(\cS),\dot{H}^s))]{S^{n-j}_{h,\Delta t}  G(X_{h,\Delta t}^{j}) I_{h} R_{\cS \to \cD} I_{h'}}^2.
		\end{split}
	\end{equation}
	For the last of these terms, we use~\eqref{eq:disc-norm-equiv}, \eqref{eq:disc-semi-smooth}, Lemma~\ref{lem:hs-interpolant-bound} and~\eqref{eq:g-lipschitz} to bound it by
	\begin{align*}
		&\Delta t\sum^{n-1}_{j=0}  \norm[L^p(\Omega,\cL_2(H_q(\cS),H))]{A_h^{\frac{s}{2}} S^{n-j}_{h,\Delta t}  G(X_{h,\Delta t}^{j}) I_{h} R_{\cS \to \cD} I_{h'}}^2 \\
		&\quad\lesssim \Delta t\sum^{n-1}_{j=0}  t_{n-j}^{-s} \norm[L^p(\Omega,\cL_2(H_q(\cS),H))]{G(X_{h,\Delta t}^{j}) I_{h} R_{\cS \to \cD} I_{h'}}^2 \\
		&\quad\lesssim \Delta t\sum^{n-1}_{j=0}  t_{n-j}^{-s} \norm[L^p(\Omega,H)]{g(X_{h,\Delta t}^{j})}^2 \lesssim \Delta t\sum^{n-1}_{j=0}  t_{n-j}^{-s} (1+ \norm[L^p(\Omega,H)]{X_{h,\Delta t}^{j}}^2).
	\end{align*}
	The first two terms of~\eqref{eq:lem:disc-regularity:pfeq1} are similarly bounded using~\eqref{eq:disc-norm-equiv}, \eqref{eq:disc-semi-smooth} and Assumption~\ref{assumption:regularity}\ref{assumption:regularity:X0}-\ref{assumption:regularity:F}. A discrete Gronwall lemma, \cite[Lemma~A.4]{K14}, completes the proof.
\end{proof}

In our main result, we show that noise interpolation results in an additional error $(h')^{r_2}$, with a rate that depends on Assumption~\ref{assumption:q} as well as the dimension $d$. We assume the triangulations $\cT_h$ and $\cT_{h'}$ to be quasiuniform, but see Remark~\ref{rem:quasi}.

\begin{theorem}
	\label{thm:main}
	Let $q$ satisfy Assumption~\ref{assumption:q} for some $\gamma > 0, \mu > 0, \psi > d/\mu$ with $$\mu_d := \mu - \max(d/2-1,d/\psi-1,0).$$ Let $X_0$ and $F$ satisfy Assumption~\ref{assumption:regularity}\ref{assumption:regularity:X0}-\ref{assumption:regularity:F} for $p \ge 2, r = r_1 \in [0,\min(\gamma,\theta))$ with $\theta$ as in \eqref{eq:sobolev_id_theta_def}, and let $G$ in \eqref{eq:G-def} satisfy \eqref{eq:g-lipschitz}. Suppose $\cT_h$ and $\cT_{h'}$ are quasiuniform and let $h \le C h'$ with $C < \infty$. For all $r_2 \in [0,2]\cap[0,\mu_d)$, there is a $C < \infty$ such that 
	\begin{equation*}
		\sup_{n \in \{1, \ldots, N_{\Delta t}\}} \norm[L^p(\Omega,H)]{X(t_n) - X_{h,\Delta t}^n} \le C (h^{1+r_1} + \Delta t^{1/2} + (h')^{r_2}) \quad  \Delta t, h, h' \in (0,1].
	\end{equation*}
\end{theorem}

\begin{proof}[Proof of Theorem~\ref{thm:main}]
	Note, that if $\mu_d \le 0$, which can only occur in $d=3$, the error is simply bounded by regularity of~\eqref{eq:mild-solution} and Lemma~\ref{lem:disc-regularity}. Below, we let $\mu_d > 0$.
	
	We make the split $X(t_n) - X_{h,\Delta t}^n = \mathrm{\cA}^n + \mathrm{\cN}^n$, where
	\begin{align*}
		\cA^n :=& \cA_1^n + \cA_2^n + \cA_3^n \\
		:=& S(t_n) X_0 - S^n_{h,\Delta t}  X_0 \\
		\quad&+ \sum^{n-1}_{j=0} \int_{t_j}^{t_{j+1}} S(t_n - s) F(X(t_j)) - S^{n-j}_{h,\Delta t}  F(X^{j}_{h,\Delta t}) \dd s \\
		\quad&+ \sum^{n-1}_{j=0} \int_{t_j}^{t_{j+1}} \Big(S(t_n - s) G(X(t_j)) - S^{n-j}_{h,\Delta t}  G(X_{h,\Delta t}^{j})\Big) R_{\cS \to \cD} \dd \tilde W(s)
	\end{align*}
	is the SPDE approximation error and 
	\begin{equation*}
		\mathrm{\cN}^n := \sum^{n-1}_{j=0} S^{n-j}_{h,\Delta t}  G(X_{h,\Delta t}^{j}) (R_{\cS \to \cD} - I_h R_{\cS \to \cD} I_{h'}) \Delta \tilde W^j
	\end{equation*}
	is the noise discretization error. The latter term is the main focus of this proof. The SPDE approximation error $\cA^n$ is treated in the exact same way as in \cite[Theorem~3.14]{K14}, wherefore we only sketch the main arguments. 
	
	The term $\cA_1^n$ is handled by making use of \eqref{eq:disc-semi-error}. One also requires regularity of $X_0$, which is Assumption~\ref{assumption:regularity}\ref{assumption:regularity:X0}. This results in the term being bounded by a constant times $h^{1+r_1} + \Delta t^{(1+r_1)/2}$. For terms $\cA_2^n$ and $\cA_3^n$, one splits up the integrals and makes use of \eqref{eq:disc-semi-error} along with similar bounds. The term $\cA_3^n$ in particular is treated by an integrated version of the bound (see \cite[Lemma 3.13]{K14}), which requires stability of $\dot{P}_h$ in $\dot{H}^1$. This is \cite[Assumption~3.5]{K14}, which is satisfied by quasiuniformity of $\cT_h$. Bounding these terms requires Lipschitz regularity of $F$ and $G$, which is Assumption~\ref{assumption:regularity}\ref{assumption:regularity:F}-\ref{assumption:regularity:G} (confirmed through Propositions~\ref{prop:F-reg} and \ref{prop:G-reg}). One also makes use of \eqref{eq:disc-semi-smooth}, spatial and temporal regularity of the solution $X$ to~\eqref{eq:spde} (see \cite[Theorems 2.27 and 2.31]{K14}) and for $\cA_3^n$ also \cite[Theorem~4.37]{DPZ14}. Altogether this results in the two terms being bounded by a constant times $h^{1+r_1} + \Delta t^{1/2}$ and a sum which can be handled by a discrete Gronwall inequality (see~\cite[Lemma~A.4]{K14}). Here the restriction on the temporal rate comes from the H\"older regularity of $X$ and only applies in the multiplicative noise setting. We refer to \cite{K14} for the complete argument for $\cA^n$. 
	
	For the bound on $\mathrm{\cN}^n$, we make use of the entropy numbers introduced in Section~\ref{sec:operator-theory}.  
	Specifically, after also applying \cite[Theorem~4.37]{DPZ14}, \eqref{eq:HS-entropy} implies that
	\begin{equation*}
		\norm[L^p(\Omega,H)]{\mathrm{\cN}^n}^2 \lesssim \Delta t \sum^{n-1}_{j=0} \E\Big[\norm[\cE^2(H_q(\cS),H)]{S^{n-j}_{h,\Delta t}  G(X_{h,\Delta t}^{j}) \big(R_{\cS \to \cD} - I_h R_{\cS \to \cD} I_{h'}\big)}^p\Big]^{\frac{2}{p}}.
	\end{equation*}
	This splits into $\norm[L^p(\Omega,H)]{\cN^n}^2 \lesssim \norm[L^p(\Omega,H)]{\cN_1^n}^2 + \norm[L^p(\Omega,H)]{\cN_2^n}^2 + \norm[L^p(\Omega,H)]{\cN_3^n}^2$ by 
	\begin{equation*}
		R_{\cS \to \cD} - I_h R_{\cS \to \cD} I_{h'} = R_{\cS \to \cD} (I - I_{h'}) + (I-I_h) R_{\cS \to \cD} - (I-I_h) R_{\cS \to \cD} (I - I_{h'}).
	\end{equation*}
	The three terms are treated similarly. We only consider the most complex case,
	\begin{equation*}
		\norm[L^p(\Omega,H)]{\cN_3^n}^2 := \Delta t \sum^{n-1}_{j=0} \E\Big[\norm[\cE^2(H_q(\cS),H)]{S^{n-j}_{h,\Delta t}  G(X_{h,\Delta t}^{j}) \big((I-I_h) R_{\cS \to \cD} (I - I_{h'})\big)}^p\Big]^{\frac{2}{p}}.
	\end{equation*}
	We start by assuming $d>1$, and derive three preliminary bounds. 
	
	First, commutativity of $S^{n-j}_{h,\Delta t}$ and $A_h$, \eqref{eq:entropy-holder}, \eqref{eq:disc-norm-equiv} and~\eqref{eq:sobolev-entropy-decay} yield for $\alpha > d$
	\begin{equation}
		\label{eq:thm:main:pf:step1}
		\begin{split}
			&\Delta t \sum^{n-1}_{j=0} \norm[\cE^\alpha(H,H)]{S^{n-j}_{h,\Delta t}}^2 \\
			&\quad= \Delta t \sum^{n-1}_{j=0} \norm[\cE^\alpha(H,H)]{A_h^\frac{r}{2} S^{n-j}_{h,\Delta t}A_h^{-\frac{r}{2}}}^2 \\
			&\quad\le \Delta t \sum^{n-1}_{j=0} t_{n-j}^{-r}  \norm[\cL]{I_{H^{r-r'} \hookrightarrow H}}^2  \norm[\cE^\alpha]{I_{H^r \hookrightarrow H^{r-r'}}}^2
			\norm[\cL]{I_{\dot{H}^r \hookrightarrow H^r}}^2 
			\norm[\cL(H,\dot{H}^r)]{A_h^{-\frac{r}{2}}\dot{P}_h}^2  \\ 
			&\quad\lesssim \norm[\cE^\alpha]{I_{H^r \hookrightarrow H^{r-r'}}}^2 \le C,
		\end{split}
	\end{equation}
	where $C < \infty$ is independent of $h, \Delta t$, by choosing $r' < r$ sufficiently close to $1$. Second, by applying the H\"older inequality, a Sobolev embedding \cite[Theorem~3.8]{BH21} and the Lipschitz continuity~\eqref{eq:g-lipschitz} in the norm~\eqref{eq:sobolev-slobodeckij} for $u \in \dot{H}^r, v \in L^\beta(\cD)$, 
	\begin{equation*}
		\norm[H]{G(u)v} \lesssim \norm[L^\rho(\cD)]{g(u)}\norm[L^\beta(\cD)]{v} \lesssim \norm[H^{r}]{g(u)}\norm[L^\beta(\cD)]{v} \lesssim \norm[\dot{H}^{r}]{u}\norm[L^\beta(\cD)]{v}, 
	\end{equation*}
	for $r \to 1$, with $\rho < \infty$ in $d=2$ and $\rho < 6$ in $d=3$, hence $\beta > d$. Thus, by Lemma~\ref{lem:disc-regularity},
	\begin{equation}
		\label{eq:thm:main:pf:step2}
		\sup_{j,\Delta t, h, h'}\norm[L^p(\Omega,\cL(L^\beta,H))]{G(X_{h,\Delta t}^{j})} \lesssim \sup_{j,\Delta t, h, h'} \norm[L^p(\Omega,\dot{H}^r)]{X_{h,\Delta t}^{j}} \lesssim 1, \quad \beta > d.
	\end{equation}
	Third, if $r_2 \in [1 + 1/\delta,2]$, Proposition~\ref{prop:interpolation} and $h \le C h'$ yield for small $\epsilon > 0$
	\begin{equation}
		\label{eq:thm:main:pf:step3}
		\begin{split}
			&\norm[\cL(W^{r_2,\delta}(\cS),L^\beta(\cD))]{(I-I_h) R_{\cS \to \cD} (I - I_{h'})} \\
			&\quad\le \norm[\cL]{I_{L^\delta \hookrightarrow L^\beta}} \norm[\cL(W^{1+1/\delta - \epsilon,\delta},L^\delta)]{I-I_h} \\
			&\hspace{2em} \times \norm[\cL(W^{1+1/\delta - \epsilon,\delta}(\cS),W^{1+1/\delta - \epsilon,\delta}(\cD))]{R_{\cS \to \cD}} \norm[\cL(W^{r_2,\delta},W^{1+1/\delta - \epsilon,\delta})]{I-I_{h'}} \\ &\quad\lesssim h^{1+1/\delta - \epsilon} (h')^{r_2 - 1 - 1/\delta + \epsilon} \lesssim (h')^{r_2}, \quad \delta \ge \beta > d.
		\end{split}
	\end{equation}
	This is also obtained for $r_2 \in (d/\delta,1 + 1/\delta)$ by instead using $\norm[\cL(W^{r_2,\delta})]{I-I_{h'}} \lesssim 1$.
	
	We now use \eqref{eq:entropy-holder}, \eqref{eq:thm:main:pf:step1}, \eqref{eq:thm:main:pf:step2} and \eqref{eq:thm:main:pf:step3} to bound~$\norm[L^p(\Omega,H)]{\cN_3^n}^2$ by 
	\begin{align*}
		&\Delta t \sum^{n-1}_{j=0} \norm[\cE^\alpha(H,H)]{S^{n-j}_{h,\Delta t}}^2 \norm[L^p(\Omega,\cL(L^\beta,H))]{G(X_{h,\Delta t}^{j})}^2 \norm[\cL(W^{r_2,\delta},L^\beta)]{(I-I_h) R_{\cS \to \cD} (I - I_{h'})}^2 \\ &\hspace{3em}\times\norm[\cE^{\frac{2\alpha}{\alpha-2}}]{I_{W^{\mu,\psi} \hookrightarrow W^{r_2,\delta}}}^2 \norm[\cL]{I_{H_q \hookrightarrow W^{\mu,\psi}}}^2 \lesssim \norm[\cE^{\frac{2\alpha}{\alpha-2}}]{I_{W^{\mu,\psi} \hookrightarrow W^{r_2,\delta}}}^2 (h')^{2 r_2}.
	\end{align*}
	Without loss of generality, let $\mu,r_2 \notin \N$. Then we can apply~\eqref{eq:sobolev-entropy-decay} if $r_2 < \mu - \max(d/\psi - d/\delta,0)$, which is true by assumption for sufficiently small $\delta \ge \beta > d$. Moreover, $\norm[\cE^{\frac{2\alpha}{\alpha-2}}]{I_{W^{\mu,\psi} \hookrightarrow W^{r_2,\delta}}} < \infty$ if $r_2 < \mu - d(\alpha-2)/2\alpha$, which holds if we let $\alpha \to d$. This completes the proof for $d > 1$. The case $d=1$ follows by similar arguments, except we need, for small $\epsilon > 0$, use the fact that $L^1(\cD) \hookrightarrow \dot{H}^{-1/2-\epsilon}$ when deriving~\eqref{eq:thm:main:pf:step2} and replace~\eqref{eq:thm:main:pf:step1} with a bound on $\norm[\cE^\alpha(\dot{H}^{-1/2-\epsilon},H)]{S^{n-j}_{h,\Delta t}}$ with $\alpha > 2$. 
\end{proof}

\begin{remark}
	\label{rem:quasi}
	By Proposition~\ref{prop:interpolation}, quasiuniformity on $\cT_{h'}$ can be replaced by regularity for $\mu_d > 1$. The same holds for $\cT_h$ if we also replace $r_1$ by $r_1 - \epsilon$, $\epsilon > 0$, using~\cite[Lemma~3.12(ii)]{K14} in place of~\cite[Lemma~3.13(i)]{K14} in~\cite[Theorem~3.14]{K14}.
\end{remark}

\section{Example kernels and numerical simulation}
\label{sec:simulations}

In this section we explore the implications of Theorem~\ref{thm:main} for several common covariance kernels. We illustrate our discussion with numerical simulations in FEniCS.

\subsection{Example kernels} We first examine for which $\gamma \in (0,1]$, $ \mu > 0 $ and $\psi > d/\mu$ Assumption~\ref{assumption:q} is satisfied for a given kernel $q$.

\begin{example}[Mat\'ern kernel]
	\label{ex:matern}
	The stationary Mat\'ern kernel is defined as
	\begin{equation*}
		q(x) := \sigma^2 2^{1-\nu}/\Gamma(\nu) ({\sqrt{2\nu}|x|}/{\rho})^\nu K_\nu(\sqrt{2\nu}|x|/\rho), \quad x \in \R^d.
	\end{equation*} 
	Here $K_\nu$ is the modified Bessel function of the second kind with smoothness parameter $\nu > 0$ and $\Gamma$ is the Gamma function. It generalizes the kernel $q_r$ of Section~\ref{sec:fractional-sobolev-spaces} by scaling with $\sigma, \rho > 0$. There, we had $H_{q} = H^{r}$ where now $r =\nu + d/2$ in terms of $\nu$, with equivalent norms. This relies on $\hat{q}$ being bounded from above and below by $\xi \mapsto C (1 + |\xi|^2)^{-r}$ for some $C \in (0,\infty)$. It is the upper bound that yields $H_{q} \hookrightarrow H^{r}$. This does not change by scaling $\sigma$ and $\rho$, and neither does the H\"older property \eqref{eq:matern-kernel-holder}. Therefore Assumption~\ref{assumption:q} is satisfied with $\psi=2, \mu = \nu + d/2$ and $\gamma = \nu$ for $\nu \in (0,1)$, $\gamma < 1$ for $\nu = 1$ and $\gamma = 1$ for $\nu > 1$.	The special case $\nu = 1/2$ results in the exponential kernel $q(x) = \sigma^2 \exp(-|x|/\rho), x \in \R^d$, see \cite[Example~7.17]{LPS14}.
\end{example}

\begin{example}[Gaussian kernel]
	\label{ex:gaussian}
	The stationary kernel $q(x) := \sigma^2 \exp(-(|x|/\rho)^2)$, for $x \in \R^d, \sigma, \rho > 0,$ fulfills Assumption~\ref{assumption:q} with $\gamma = 2$ ,$\psi = 2$ and any $\mu > d/2$. The latter follows from $\hat{q}(\xi) = \sigma^2(2\rho)^{-2} \exp(-\rho^2|\xi|^2/4), \xi \in \R^d$, cf.\ Example~\ref{ex:matern}.
\end{example}

\begin{example}[Polynomial kernels with compact support]
	\label{ex:polynomial}
	Let $p \colon \R \to \R$ be a polynomial and let the stationary kernel $q$ be of the form
	\begin{equation*}
		q(x) := \begin{cases}
			p(|x|) \text{ if } |x| \in [0,c], \\
			0 \text{ otherwise,} 
		\end{cases} \quad x \in \R^d.
	\end{equation*}
	Not all $p$ yield positive semidefinite $q$. We consider two admissible examples for which one can use Proposition~5.5 and Theorem~5.26 in \cite{W04} (to which we refer for more examples) to compute and then bound $\hat{q}$ by a constant in a neighbourhood of $0$. Without loss of generality we let $c = 1$.
	
	First, \cite[Theorem~6.20]{W04} shows that $p(x) := (1-x)^2$ yields a positive definite $q$ . Assumption~\ref{assumption:q}\ref{assumption:q-holder} holds with $\gamma = 1/2$ as $p(0) - p(x) = 2x - x^2$ in $[0,1]$. Moreover, by \cite[Lemmas~10.31, 10.34]{W04}, $\hat{q}$ is bounded by $\xi \mapsto C|\xi|^{-(d+1)}$ for some $C < \infty$. Thus, Assumption~\ref{assumption:q}\ref{assumption:q-sobolev} holds for $\psi = 2$ and $\mu = 1/2+d/2$.	
	
	Second, letting $p(x) := (1-x)^4(4x+1)$ yields a positive definite $q$ by \cite[Theorem~9.13]{W04}. Assumption~\ref{assumption:q}\ref{assumption:q-holder} holds with $\gamma = 1$ due to the derivative properties $p'(0) = 0$ and $p''$  bounded in $[0,1]$. Moreover, \cite[Theorem~10.35]{W04} for $d\ge 2$ and a direct calculation for $d=1$ bounds $\hat{q}$ by $\xi \mapsto C|\xi|^{-(d+3)}$, so Assumption~\ref{assumption:q}\ref{assumption:q-sobolev} holds with $\psi = 2,\mu = 3/2+d/2$.
\end{example}

The next example shows why we cannot always have $\psi = 2$ in Assumption~\ref{assumption:q}.

\begin{example}[Factorizable kernels]
	\label{ex:factorizable-matern}
	A kernel $q$ is called factorizable if it is a product $q(x,y) := \prod_{j=1}^d q_j(x_j,y_j)$, $x = (x_1,\ldots,x_d), y = (y_1,\ldots,y_d) \in \R^d, d\ge 2$, of univariate kernels. Let $q_1 := \cdots = q_d := \tilde{q}$ be Matérn kernels, all with $\nu > 0$, so that $q$ is stationary. From equivalence of norms in $\R^d$ and the equality
	\begin{equation*}
		q(0)-q(x-y) = \tilde q(0)^d - \prod_{j=1}^d \tilde q(x_j-y_j) = \sum_{j=1}^d \Big( \tilde q(0)^{d-j} \prod_{i=1}^{j-1} \tilde q(x_i-y_i)\Big) (\tilde q(0)-\tilde q(x_j-y_j)) 
	\end{equation*}
	it follows that $\gamma = \nu$ in Assumption~\ref{assumption:q}\ref{assumption:q-holder}. From \cite[Example~7.12]{LPS14}, $\hat{q}$ is given by
	\begin{equation*}
		\hat{q}(\xi) = C \prod_{j=1}^d (1+|\xi_j|^2) \lesssim (1+|\xi|^2)^{-\nu - 1/2}, \quad \xi = (\xi_1, \ldots, \xi_d) \in \R^d, C \in (0,\infty).
	\end{equation*}
	Using this result in~\cite[Theorem~10.12]{W04}, it can be seen that $H_q \hookrightarrow H^{1/2 + \nu}$. If $\nu \le 1/2$, this embedding does not satisfy Assumption~\ref{assumption:q}\ref{assumption:q-sobolev} with $\psi=2$ since $1 + 2 \nu \le d$. However, we can find $\mu < \nu + 1/2$ and $\psi > 2$ such that $\mu \psi > d$. We let the factorizable exponential kernel $q(x,y) = \sigma^2 \exp(-(|x_1-y_1|+|x_2-y_2|)/\rho)$ in $d=2$ illustrate this. Here, $\nu = 1/2$ so $H_q \hookrightarrow H^1$. By the reproducing property, we find that for $x,y \in \cS$,
	\begin{equation*}
		|v(x) - v(y)|^2 = |\inpro[H_q]{v}{q(x,\cdot)-q(y,\cdot)}|^2 \le \norm[H_q]{v}^2 2(q(0)-q(x-y)) \lesssim \norm[H_q]{v}^2 |x-y|.
	\end{equation*}
	This implies that $H_q \hookrightarrow W^{(1-\epsilon)/2,p'}$ for all $p' > 1$ and $\epsilon > 0$ by definition of $\norm[W^{(1-\epsilon)/2,p'}]{\cdot}$. Combining these two embeddings, there is for all $\mu \in (1/2,1)$ a $\psi > 2$ such that $\mu \psi > 2$ and $H_q \hookrightarrow W^{\mu,\psi}$, satisfying Assumption~\ref{assumption:q}\ref{assumption:q-sobolev}. To see this, let $p' > 4$ and pick $\epsilon > 0$ so small that $\mu \in ((1-\epsilon)/2,1-\epsilon)$. Then $W^{\mu,\psi} = [W^{(1-\epsilon)/2,p'},W^{1-\epsilon,2}]_\theta \hookleftarrow H_q$ in terms of complex interpolation with $\psi = 2p'/(2(1-\theta)+p'\theta)$ and $\theta = 2 \mu/(1-\epsilon) - 1 \in (0,1)$, so that $\mu \psi > 2$ if we let $\epsilon < (p'-4)(1-\mu)/(p'(\mu+1)-4)$. 
\end{example}

Lastly, we show how the classical setting of~\cite[Section~4]{JR12} fits in Assumption~\ref{assumption:q}.

\begin{example}[Kernel with known Mercer expansion]
	\label{ex:kernel-basis}
	Suppose there is an ONB $(f_j)_{j=1}^\infty$ of $L^2(\cS)$ of uniformly bounded continuous functions, with, for some $r \in (0,1]$,
	\begin{equation*}
		\sum_{j=1}^{\infty} \mu_j \sup_{x \neq y \in \cS} \frac{(f_j(x) - f_j(y))^2}{|x-y|^{2r}} < \infty \quad \text{ with } \mu_j\ge0, j \in \N, \text{ and } (\mu_j)_{j=1}^\infty \in \ell^1.
	\end{equation*}
	Let the kernel $q$ be given by $q(x,y) := \sum_{j=1}^\infty \mu_j f_j(x) f_j(y)$ for $x,y \in \cS$ and define a Hilbert space $H_q$ by
	\begin{equation*}
		H_q := \Big\{v \in L^2(\cS): \sum_{j=1}^\infty \mu_j^{-1} |\inpro[L^2]{v}{f_j}|^2  < \infty \Big\}, \quad \inpro[H_q]{u}{v} := \sum^\infty_{j=1} \mu_j^{-1} \inpro[L^2]{v}{f_j}\inpro[L^2]{u}{f_j}
	\end{equation*}
	for $u,v \in H_q$, under the notational convention $0/0=0$. Since $q(x,\cdot) \in H_q$ and $v(x)=\inpro[H_q]{v}{q(x,\cdot)}$ for $v \in H_q, x \in \cD$, $H_q$ must be the RKHS for $q$ on $\cS$. We have
	\begin{equation*}
		\frac{q(x,x) - 2q(x,y) + q(y,y)}{|x-y|^{2r}} = \sum_{j=1}^{\infty} \mu_j \frac{(f_j(x) - f_j(y))^2}{|x-y|^{2r}} \le \sum_{j=1}^{\infty} \mu_j \sup_{x \neq y \in \cS} \frac{(f_j(x) - f_j(y))^2}{|x-y|^{2r}},
	\end{equation*}
	meaning that Assumption~\ref{assumption:q}\ref{assumption:q-holder} holds for $r = \gamma$. As in Example~\ref{ex:factorizable-matern}, the reproducing property yields that $\mu \in (0,r] \cap (0,1)$ and  $\psi > d/\mu$ in Assumption~\ref{assumption:q}\ref{assumption:q-sobolev}.
\end{example}

\subsection{Numerical simulations}
We approximate the strong error~\eqref{eq:intro-strong-error} with $T = 1$, $p=2$, $\cD$ a regular dodecagon centered at $(0.5,0.5)$ with radius $0.5$ and $\cS$ the unit square. We let $A := 10^{-2}(-\Delta + 1)$ with either Dirichlet or Neumann zero boundary conditions. We set $G$ and $F$ according to \eqref{eq:G-def} and \eqref{eq:F1-F2-def}, picking different functions $b, f$, and $g$. We fix $\Delta t = 10^{-3}$ and let $h = 2^{-1}, \ldots, 2^{-5}$ with $h' = h$ (except in Example~\ref{ex:numerical-diffgrid}). The time step size and the scaling of $A$ has been chosen so that, when plotting the error versus $h$ in a log-log scale, we see the asymptotic spatial convergence rate appear. This comes from either the SPDE approximation error or the noise discretization error depending on which one dominates (whether $r_1 + 1 \le r_2$ or not). We write $r \epel s$, $r,s \in \R$, to denote $r = s - \epsilon$ for arbitrary $\epsilon > 0$.

We compute a Monte Carlo approximation of the strong error, given by
\begin{equation*}
	\sup_{j \in \{1, \ldots, N_{\Delta t}\}} \left( \frac{1}{M} \sum_{i=1}^M \norm[H]{X_{h,\Delta t,(i)}^j - X_{(i)}(t_j)}^2\right)^{\frac 1 2}
\end{equation*}
for $M=80$ independent samples $X_{h,\Delta t,(i)}$ and $X_{(i)}$ of $X_{h,\Delta t}$ and $X$. In practice we replace $X$ by a reference solution $X_{h,\Delta t}$ with $h = 2^{-7}$. This sample size choice does not yield a precise estimate of~\eqref{eq:intro-strong-error}. However, our interest is only in observing the strong convergence rate. For this, the Monte Carlo error causes no problems since it is itself bounded by a constant times the strong error, for fixed $t_j$. This follows from the $1/2$-H\"older continuity of the square root and the triangle inequality via
\begin{align*}
	&\E\Big[\Big| \Big( \frac{1}{M} \sum_{i = 1}^M \norm[H]{X_{h,\Delta t,(i)}^j - X_{(i)}(t_j)}^2\Big)^{\frac 1 2} - \E[\norm[H]{X_{h,\Delta t}^j - X(t_j)}^2]^{\frac 1 2} \Big|^2\Big] \\
	&\quad\le \E\Big[\Big| \frac{1}{M} \sum_{i = 1}^M \big(\norm[H]{X_{h,\Delta t,(i)}^j - X_{(i)}(t_j)}^2 - \E[\norm[H]{X_{h,\Delta t}^j - X(t_j)}^2]\big) \Big|\Big] \\
	&\quad\le 2 \E[\norm[H]{X_{h,\Delta t}^j - X(t_j)}^2].
\end{align*}
We also use the same $M$ samples of $\tilde W$ across all values of $h$ to further decrease the influence of the Monte Carlo error.

The simulations are performed with the FEniCS software package (see~\cite{LL16}) on a laptop with an Intel\textsuperscript{\tiny\textregistered} Core\textsuperscript{\tiny\texttrademark} i7-10510U processor, with $\tilde W$ sampled by circulant embedding \cite{GKNSS18}. The code accompanies the arXiv version of the paper.

\begin{figure}
	\centering		
	\subfigure[Example~\ref{ex:numerical-matern}, Matérn kernel.\label{subfig:matern}]{\includegraphics[width = 0.32\textwidth]{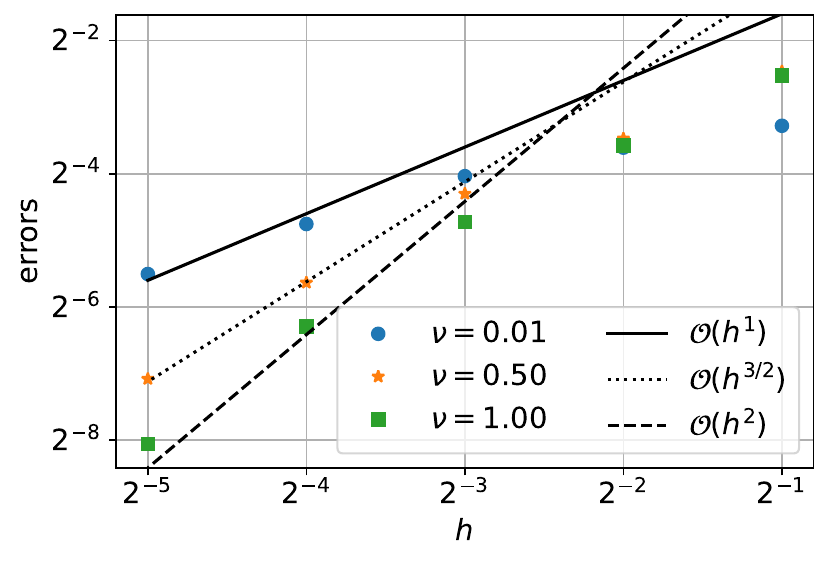}}
	\subfigure[Example~\ref{ex:numerical-sepexp}, two exponential type covariances. \label{subfig:sepexp}]{\includegraphics[width = 0.32\textwidth]{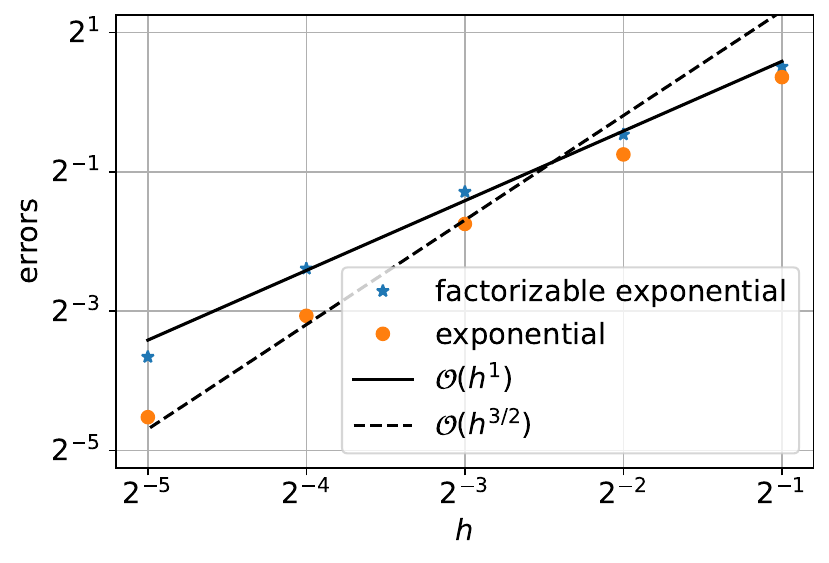}}
	\subfigure[Example~\ref{ex:numerical-diffgrid}, polynomial kernel with different resolutions of $\cT_{h'}$. \label{subfig:diffgrid}]{\includegraphics[width = 0.32\textwidth]{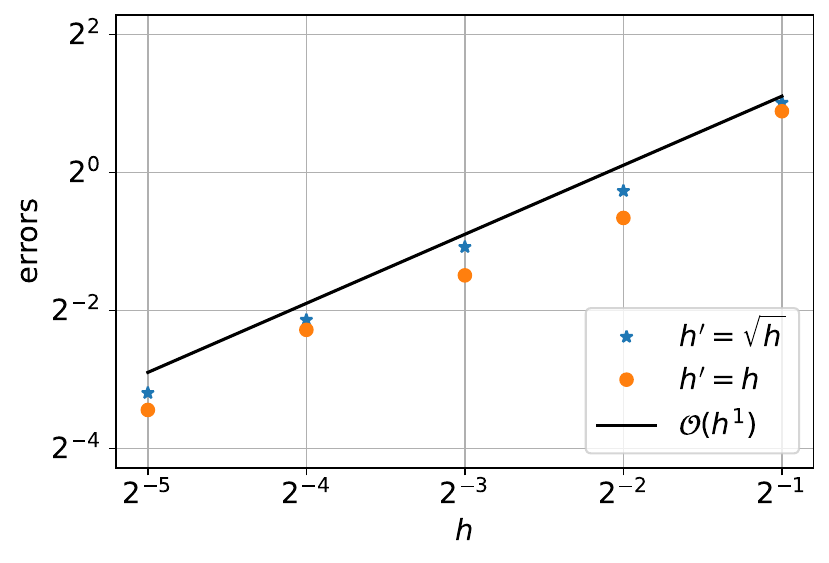}}
	\caption{Strong error approximations for $X_{h,\Delta t}$.}
\end{figure} 

\begin{example}[Non-dominance of noise discretization errors]
	\label{ex:numerical-matern}	
	In the case that $q$ is a Matérn kernel as in Example~\ref{ex:matern}, we have $r_1 + 1 \le r_2 \epel \min(2,1 + \nu)$ for $d \in \{2, 3\}$, i.e., the noise discretization error does not dominate. This is true also for the Gaussian kernel of Example~\ref{ex:gaussian} and any $d$. It is also true for the first polynomial kernel of Example~\ref{ex:polynomial} when $d \in \{2, 3\}$ (since then $r_1 +1 = r_2 \epel 3/2$) and the second for any $d$ (since then $r_1 + 1 \le r_2 \epel 2$). 
	
	We illustrate this in the simulation setting outlined above with $q$ of Matérn type with parameters $\rho = 0.25, \sigma^2 = 10$ and $\nu \in \{0.01,0.5,1.0\}$. We let $X_0 = b(x) = 0$, pick Lipschitz functions $f(u,x) = f(u) = 10^{-1} + u/(|u|+1)$, $g(u,x) = g(u) = u/(|u|+1)$ for $u \in \R, x \in \cD$, set $h' = h$ and apply Neumann boundary conditions. The errors of Figure~\ref{subfig:matern} are in line with the results of Theorem~\ref{thm:main}. See Figure~\ref{subfig:realization} for a realization of $X_{h,\Delta t}$ with $\nu = 0.5, h = 2^{-7}$ and $\Delta t = 10^{-3}$.
\end{example}

\begin{example}[Dominance of noise discretization errors]
	\label{ex:numerical-sepexp}
	In $d=1$, the Matérn kernels yield $r_1 + 1 \epel \min(\nu,\theta,1) +1$ and $r_2 \epel \min(\nu + 1/2,2)$, with dominant noise discretization error if $\nu \le 1$ for Dirichlet and $\nu \le 3/2$ for Neumann conditions. Similar conclusions hold for the first polynomial kernel of Example~\ref{ex:polynomial}. This was demonstrated in a linear setting in~\cite{BDLP22}, with different rates owing to errors being measured pointwise in space. Numerical experiments (not included in~\cite{BDLP22}) suggest this applies in our $L^2$ setting as well.
	
	Factorizable kernels (Example~\ref{ex:factorizable-matern}) may exhibit dominant noise discretization errors even when $d > 1$. We illustrate this with a numerical simulation comparing the exponential and the factorizable exponential kernel, both with $\rho = 0.25$ and $\sigma^2 = 10$. We set $f(u) = 10^{-1}$, $g(u) = 1$ for $u \in \R$, $b(x) = 10^{-1}(1,1)$ for $x \in \cD$, equip $A$ with Dirichlet boundary conditions and otherwise follow the setting of Example~\ref{ex:numerical-matern}. From Examples~\ref{ex:matern} and~\ref{ex:factorizable-matern} we see that $r_1 + 1 \epel 3/2$ for both the exponential and the factorizable exponential kernel, whereas $r_2 \epel 3/2$ in the former case and $r_2 \epel 1$ in the latter. The results (Figure~\ref{subfig:sepexp}) align with Theorem~\ref{thm:main}.
	
	One can potentially mitigate these issues by computing $I_h R_{\cS \to \cD} I_{h'} \Delta \tilde W^j$ using finer meshes on $\cS$ and $\cD$ than the mesh used in~\eqref{eq:spde-approximation}. Efficient sampling methods could still maintain competitiveness of the approximation schemes. However, as FEniCS does not appear to support this, we do not explore this further.
\end{example} 

\begin{example}[Dominance of SPDE approximation error and noise sampling on a coarse mesh]
	\label{ex:numerical-diffgrid}
	In our last example, we consider a non-smooth initial condition, namely $X_0(x) = 3(-\log((x_1-0.5)^2 + (x_2-0.5)^2))^{1/3}$, $ x = (x_1, x_2) \in \cD$. This function, while in $H^1(\cD)$, is not in $H^{1+\epsilon}(\cD)$ due to its discontinuity (cf.\ \cite[page~1]{G85}).
	With $q$ as the second polynomial kernel of Example~\ref{ex:polynomial}, $\sigma^2=10$, $f(u) = 10^{-1}$, $g(u) = 1$, $b(x) = 0$, and Neumann boundary conditions for $A$, we find $r_1 + 1 \epel 1$ and $r_2 \epel 2$. Hence, the error stemming from the initial condition dominates the noise discretization error. We can therefore set $h' = h^{1/2}$ and retain the same convergence rate as $h' = h$, which is seen in Figure~\ref{subfig:diffgrid}. Thus, for rough initial conditions, the computational effort can be further decreased without any essential loss of accuracy. Similarly, if we had a smooth initial condition with Dirichlet boundary conditions for $A$, we could set $h' = h^{3/4}$ for analogous, albeit smaller, computational savings. This underscores the importance of noise discretization error analysis for efficient algorithm design.
\end{example}

\section*{Acknowledgments}
We would like to express many thanks to two anonymous referees who helped improve the results and presentation. 

\bibliographystyle{siamplain}
\bibliography{interpolation-spde}

\begin{thebibliography}{10}

\bibitem{ANZ98}
{\sc E.~J. Allen, S.~J. Novosel, and Z.~Zhang}, {\em Finite element and
  difference approximation of some linear stochastic partial differential
  equations}, Stochastics Stochastics Rep., 64 (1998), pp.~117--142.

\bibitem{BL12c}
{\sc A.~Barth and A.~Lang}, {\em Simulation of stochastic partial differential
  equations using finite element methods}, Stochastics, 84 (2012),
  pp.~217--231.

\bibitem{BH21}
{\sc A.~Behzadan and M.~Holst}, {\em Multiplication in {S}obolev spaces,
  revisited}, Ark. Mat., 59 (2021), pp.~275--306.

\bibitem{BB01}
{\sc F.~Ben~Belgacem and S.~C. Brenner}, {\em Some nonstandard finite element
  estimates with applications to {{\(3D\)}} {Poisson} and {Signorini}
  problems}, ETNA, Electron. Trans. Numer. Anal., 12 (2001), pp.~134--148.

\bibitem{BDLP22}
{\sc F.~E. Benth, G.~Di~Nunno, G.~Lord, and A.~Petersson}, {\em The heat
  modulated infinite dimensional heston model and its numerical approximation}.
\newblock Preprint at arXiv:2206.10166, June 2022.

\bibitem{BT04}
{\sc A.~Berlinet and C.~Thomas-Agnan}, {\em Reproducing kernel {H}ilbert spaces
  in probability and statistics}, Kluwer Academic Publishers, Boston, MA, 2004.

\bibitem{BGT15}
{\sc M.~Boulakia, A.~Genadot, and M.~Thieullen}, {\em Simulation of {SPDEs} for
  excitable media using finite elements}, J. Sci. Comput., 65 (2015),
  pp.~171--195.

\bibitem{BS08}
{\sc S.~C. Brenner and L.~R. Scott}, {\em The mathematical theory of finite
  element methods}, vol.~15 of Texts in Applied Mathematics, Springer, New
  York, third~ed., 2008.

\bibitem{C13}
{\sc P.~j. Ciarlet}, {\em Analysis of the {Scott}-{Zhang} interpolation in the
  fractional order {Sobolev} spaces}, J. Numer. Math., 21 (2013), pp.~173--180.

\bibitem{CH19}
{\sc J.~Cui and J.~Hong}, {\em Strong and weak convergence rates of a spatial
  approximation for stochastic partial differential equation with one-sided
  {L}ipschitz coefficient}, SIAM J. Numer. Anal., 57 (2019), pp.~1815--1841.

\bibitem{DPZ14}
{\sc G.~Da~Prato and J.~Zabczyk}, {\em Stochastic equations in infinite
  dimensions}, vol.~152 of Encyclopedia of mathematics and its applications,
  Cambridge University Press, Cambridge, second~ed., 2014.

\bibitem{DZ02}
{\sc Q.~Du and T.~Zhang}, {\em Numerical approximation of some linear
  stochastic partial differential equations driven by special additive noises},
  SIAM J. Numer. Anal., 40 (2002), pp.~1421--1445.

\bibitem{DS80}
{\sc T.~Dupont and R.~Scott}, {\em Polynomial approximation of functions in
  {S}obolev spaces}, Math. Comp., 34 (1980), pp.~441--463.

\bibitem{ET96}
{\sc D.~E. Edmunds and H.~Triebel}, {\em Function spaces, entropy numbers,
  differential operators}, vol.~120 of Camb. Tracts Math., Cambridge: Cambridge
  Univ. Press, 1996.

\bibitem{FS91}
{\sc H.~Fujita and T.~Suzuki}, {\em Evolution problems}, in Evolution problems,
  P.~G. Ciarlet and J.-L. Lions, eds., Handbook of numerical analysis, II,
  North-Holland, Amsterdam, 1991, pp.~789 -- 928.
\newblock Finite element methods. Part 1.

\bibitem{GKNSS18}
{\sc I.~G. Graham, F.~Y. Kuo, D.~Nuyens, R.~Scheichl, and I.~H. Sloan}, {\em
  Analysis of circulant embedding methods for sampling stationary random
  fields}, SIAM J. Numer. Anal., 56 (2018), pp.~1871--1895.

\bibitem{G85}
{\sc P.~Grisvard}, {\em Elliptic problems in nonsmooth domains}, vol.~24 of
  Monographs and studies in mathematics, Pitman (Advanced Publishing Program),
  Boston, MA, 1985.

\bibitem{JR12}
{\sc A.~Jentzen and M.~Röckner}, {\em Regularity analysis for stochastic
  partial differential equations with nonlinear multiplicative trace class
  noise}, Journal of Differential Equations, 252 (2012), pp.~114--136.

\bibitem{KLL10}
{\sc M.~Kov{\'a}cs, S.~Larsson, and F.~Lindgren}, {\em Strong convergence of
  the finite element method with truncated noise for semilinear parabolic
  stochastic equations with additive noise}, Numer. Algorithms, 53 (2010),
  pp.~309--320.

\bibitem{KLL11}
{\sc M.~Kov{\'a}cs, F.~Lindgren, and S.~Larsson}, {\em Spatial approximation of
  stochastic convolutions}, J. Comput. Appl. Math., 235 (2011), pp.~3554--3570.

\bibitem{KLP22}
{\sc M.~Kovács, A.~Lang, and A.~Petersson}, {\em Hilbert–{S}chmidt
  regularity of symmetric integral operators on bounded domains with
  applications to {SPDE} approximations}, Stochastic Analysis and Applications,
   (2022).
\newblock Published online.

\bibitem{K14}
{\sc R.~Kruse}, {\em Strong and weak approximation of semilinear stochastic
  evolution equations}, vol.~2093 of Lecture notes in mathematics, Springer,
  Cham, 2014.

\bibitem{LL16}
{\sc H.~P. Langtangen and A.~Logg}, {\em Solving {PDEs} in {Python}. {The}
  {FEniCS} tutorial {I}}, vol.~3 of Simula SpringerBriefs Comput., Cham:
  Springer Open, 2016.

\bibitem{LPS14}
{\sc G.~J. Lord, C.~E. Powell, and T.~Shardlow}, {\em An introduction to
  computational stochastic PDEs}, Cambridge Texts in Applied Mathematics,
  Cambridge University Press, 2014.

\bibitem{P80}
{\sc A.~{Pietsch}}, {\em {Operator ideals. Licenced ed}}, vol.~20, Elsevier
  (North-Holland), Amsterdam, 1980.

\bibitem{PR07}
{\sc C.~Pr{\'e}v{\^o}t and M.~R{\"o}ckner}, {\em A {Concise} {Course} on
  {Stochastic} {Partial} {Differential} {Equations}}, vol.~1905 of Lecture
  notes in mathematics, Springer, Berlin, 2007.

\bibitem{R11}
{\sc M.~Riedle}, {\em Cylindrical {Wiener} processes}, in S\'eminaire de
  Probabilit\'es XLIII, Poitiers, France, Juin 2009., Berlin: Springer, 2011,
  pp.~191--214.

\bibitem{S99}
{\sc M.~L. Stein}, {\em Interpolation of spatial data}, Springer Series in
  Statistics, Springer-Verlag, New York, 1999.
\newblock Some theory for Kriging.

\bibitem{T06}
{\sc V.~Thom\'{e}e}, {\em Galerkin finite element methods for parabolic
  problems}, vol.~25 of Springer Series in Computational Mathematics,
  Springer-Verlag, Berlin, second~ed., 2006.

\bibitem{W05}
{\sc J.~B. Walsh}, {\em Finite element methods for parabolic stochastic
  {PDE}'s}, Potential Anal., 23 (2005), pp.~1--43.

\bibitem{W04}
{\sc H.~Wendland}, {\em Scattered data approximation}, vol.~17 of Cambridge
  Monographs on Applied and Computational Mathematics, Cambridge University
  Press, Cambridge, 2005.

\bibitem{Y10}
{\sc A.~Yagi}, {\em Abstract parabolic evolution equations and their
  applications}, Springer Monographs in Mathematics, Springer-Verlag, Berlin,
  2010.

\bibitem{Y05}
{\sc Y.~Yan}, {\em Galerkin finite element methods for stochastic parabolic
  partial differential equations}, SIAM J. Numer. Anal., 43 (2005),
  pp.~1363--1384.

\end{thebibliography}

\appendix

\section{A fractional Sobolev norm error bound for piecewise linear finite element interpolants}
\label{sec:appendix}

Here we derive the bound $\norm[W^{r,p}]{(I - I_h)v} \lesssim h^{\min(s-r,2)} \norm[W^{s,p}]{v}$ of Proposition~\ref{prop:interpolation}, generalizing the results of~\cite{BB01} to $p \neq 2$ and $s < 1$. For the Scott--Zhang interpolant, \cite{C13} treats this case, but as far as we know Proposition~\ref{prop:interpolation} cannot be inferred from this. 
\begin{proof}[Proof of Proposition~\ref{prop:interpolation}]
	In this proof, we use for general $\cD$ the notation 
	\begin{equation*}
		|u|_{W^{r,p}(\cD)}^{p} := \sum_{|\alpha| = m} \int_{\cD \times \cD} \frac{|D^\alpha u(x) - D^\alpha u(y)|^p}{|x-y|^{d+p\sigma}}\dd x \dd y \quad r = m + \sigma, m \in \N_0, \sigma \in (0,1).
	\end{equation*}
	Here $|\cdot|_{W^{r,p}(\cD)}^{p}$, $r \in \N$, is the usual Sobolev seminorm and the $L^p(\cD)$-norm for $r=0$.
	
	Without loss of generality, let $s \le 2$. We start with the cases $s \in (1,2]$ and $r = 0$ or $r \in (1,1+1/p) \cap(1,s]$. The claim for $s \in (1,2]$, $r \in [0,1]$ then follows from real interpolation, see, e.g., \cite[Theorem~3.1, Remark~3.3]{BH21}. 
	First note that for all $k \in \N_0, \lambda \in (0,1)$, there is a $C < \infty$ such that for all $w \in W^{k+\lambda,p}$,
	\begin{equation}
		\label{eq:thm:interpolation:1}
		|w|_{W^{k+\lambda,p}(\cD)}^p \le C \sum_{|\alpha|=k} \sum_{T \in \cT_h} \left( |D^\alpha w|^p_{W^{\lambda,p}(T)} + \int_T \frac{|D^\alpha w (x)|^p}{(\inf_{y \in \partial T} |x-y|)^{p\lambda}} \dd x \right). 
	\end{equation}
	This is \cite[Lemma~2.1]{BB01} for $p=2$, the proof is analogous for $p \neq 2$. We also need an estimate on the reference simplex $\hat{T}$ with vertices $\hat x_0 = (0,0,\ldots,0), \hat x_1 = (1,\ldots,0)$, \ldots $\hat x_{d} = (0,\ldots,d)$. For $\lambda \in [0,1/p)$, \cite[Theorem~1.4.4.3]{G85} yields a $C < \infty$ such that
	\begin{equation}
		\label{eq:thm:interpolation:4}
		\int_{\hat T} \frac{|w(x)|^p}{(\inf_{y \in \partial \hat T} |x-y|)^{p\lambda}} \dd x \le C \norm[W^{\lambda,p}(\hat{T})]{w}^p, \qquad w \in W^{\lambda,p}(\hat{T}).
	\end{equation}
	For each $T \in \cT_h$, there is an affine transformation $\hat{T} \ni \hat{x} \mapsto B_T x  + b_T \in T$. By regularity, there are $C_1, C_2 \in (0,\infty)$, that only depend on $\cD$ (see \cite{BB01,BS08,DS80}), such that
	\begin{equation}
		\label{eq:thm:interpolation:2}
		C_1 \mathrm{diam}(T) |\hat x| \le |B_T \hat x| \le C_2 \mathrm{diam}(T) |\hat x| \le C_2 h |\hat x|, \quad \hat x \in \hat{T}, T \in \cT_h \text{ and }
	\end{equation}
	\begin{equation}
		\label{eq:thm:interpolation:3}
		C_1 \mathrm{diam}(T)^d \le |\det B_T| \le C_2 \mathrm{diam}(T)^d \le C_2 h^d, \quad T \in \cT_h.
	\end{equation}
	For quasiuniform $\cT_h$, we also have corresponding lower bounds in terms of $h$.
	
	Writing $v_h := I_h v$, $\eqref{eq:thm:interpolation:1}$ yields $|v-v_h|_{W^{r,p}(\cD)}^p \lesssim \sum_{T \in \cT_h} \mathrm{I}_T + \mathrm{II}_T$ for $r > 1$, with  
	\begin{equation*}
		\mathrm{I}_T := |v-v_h|^p_{W^{r,p}(T)} \text{ and } \mathrm{II}_T := \sum_{|\alpha|=1} \int_T \frac{|D^\alpha v (x)-D^\alpha v_h (x)|^p}{(\inf_{y \in \partial T} |x-y|)^{p(r-1)}} \dd x.
	\end{equation*}
	In the case that $r = 0$, $\mathrm{II}_T := 0$. 
	Let us write $e_h := v - v_h$, $\hat v : = v(B_T \cdot + b_T)$ and note that $\nabla \hat v_h = B_T^* \nabla v_h$, cf.\ \cite[page~461]{DS80}. Combining this, \eqref{eq:thm:interpolation:2}, \eqref{eq:thm:interpolation:3} and equivalence of norms in $\R^d$, it follows by a scaling argument that for $r > 1$
	\begin{align*}
		\mathrm{I}_T &= |\det B_T|^2 \int_{\hat{T}} \int_{\hat{T}} \frac{\sum_{|\alpha|=1}\left|D^\alpha e_h(B_T \hat x + b_T) - D^\alpha e_h(B_T \hat y + b_T)\right|^p}{|B_T(\hat x-\hat y)|^{d + (r-1)p}} \dd \hat{x} \dd \hat{y} \\
		&\lesssim |\det B_T|^2 \int_{\hat{T}} \int_{\hat{T}} \frac{ \left|(B_T^*)^{-1}\left(\nabla \hat e_h(\hat x) - \nabla \hat e_h(\hat y)\right)\right|^p }{|B_T(\hat x-\hat y)|^{d + (r-1)p}} \dd \hat{x} \dd \hat{y} \lesssim \mathrm{diam}(T)^{d-rp} |\hat{e}_h|^p_{W^{r,p}(\hat T)}.
	\end{align*}
	For $r = 0$, $\mathrm{I}_T$ is instead bounded by $|\det B_T| \norm[L^p(\hat{T})]{\hat{e}_h}^p \lesssim \mathrm{diam}(T)^d \norm[L^p(\hat{T})]{\hat{e}_h}^p$. 
	In case $r > 1$, we also use~\eqref{eq:thm:interpolation:4} with $\lambda = r-1 < 1/p$ to similarly deduce the estimate
	\begin{align*}
		\mathrm{II}_T \lesssim \mathrm{diam}(T)^{d-rp} \Big(|\hat{e}_h|^p_{W^{1,p}(\hat T)}+|\hat{e}_h|^p_{W^{r,p}(\hat T)}\Big) \lesssim \mathrm{diam}(T)^{d-rp} \norm[W^{r,p}(\hat T)]{\hat{e}_h}^p. 
	\end{align*}
	
	Let $I_h^{\hat T}$ denote interpolation onto the set $\cP_1(\hat T)$ of linear polynomials on $\hat T$ and note that $\hat{v}_h = I_h^{\hat T} \hat{v}$. The Sobolev embedding $W^{s,p}(\hat T) \hookrightarrow L^\infty(\hat T)$ (see \cite[Theorem~1.4.4.1]{G85}) yields $|\hat{v}_h|^p_{W^{1,p}(\hat T)} \lesssim \norm[W^{s,p}(\hat T)]{\hat{v}}^p$ after using equivalence of norms in $\R^d$ to see that
	\begin{equation}
		\label{eq:thm:interpolation:5}
		|\hat{v}_h|^p_{W^{1,p}(\hat T)} \lesssim 
		\norm[L^p(\hat T)]{\nabla \hat v_h}^p 
		=(d!)^{-1} \Big(\sum_{j=1}^d |\hat{v}(\hat{x}_0)-\hat{v}(\hat{x}_j)|^2\Big)^{p/2}.
	\end{equation}
	Similarly, $\norm[L^p(\hat T)]{\hat{v}_h} \lesssim \norm[W^{s,p}(\hat T)]{\hat{v}}$.
	Since $|f|_{W^{r,p}(\hat{T})} = 0$ for $r > 1$ and $f \in \cP_1(\hat{T})$, this gives $I_h^{\hat T} \in \cL(W^{s,p}(\hat T),W^{r,p}(\hat T))$ by real interpolation. These observations, the invariance of $I_h^{\hat T}$ on $\cP_1(\hat{T})$ and the Bramble--Hilbert Lemma \cite[Theorem~6.1]{DS80} imply that
	\begin{align*}
		\norm[W^{r,p}(\cD)]{v-v_h}^p &\lesssim \sum_{T \in \cT_h} \mathrm{diam}(T)^{d-rp} \inf_{f \in \cP_1} \norm[W^{s,p}(\hat{T})]{\hat{v} - f}^p \\
		&\lesssim \sum_{T \in \cT_h} \mathrm{diam}(T)^{d-rp} |\hat{v}|_{W^{s,p}(\hat{T})}^p.
	\end{align*}
	This sum is bounded by $h^{(s-r)p}|v|^p_{W^{s,p}(\cD)}$ after applying this estimate to the summand:
	\begin{align*}
		&|\det B_T^{-1}|^2 \int_{T} \int_{T} \frac{\sum_{|\alpha|=1}\left|D^\alpha \hat v(B_T^{-1} (x - b_T)) - D^\alpha \hat v(B_T^{-1} (y - b_T))\right|^p}{|B_T^{-1}( x- y)|^{d + (s-1)p}} \dd x \dd y \\ 
		&\quad\lesssim |\det B_T^{-1}|^2 \int_{T} \int_{T} \frac{\left|B_T^*(\nabla v(x) - \nabla v(y))\right|^p}{|B_T^{-1}( x- y)|^{d + (s-1)p}} \dd x \dd y \lesssim \mathrm{diam}(T)^{sp-d}  |v|^p_{W^{s,p}(T)}.
	\end{align*} 
	
	The proof for $s = 1$ follows from interpolation of classical estimates, see~\cite[Remark~4.4.27]{BS08}. The proof for $s < 1$ and $r=0$ is similar to the case $r=0$ and $s > 1$, except we apply~\eqref{eq:thm:interpolation:1} with $k = 0$ and bound the infimum over $\cP_1(\hat T)$ when applying \cite[Theorem~6.1]{DS80} with the infimum over $\cP_0(\hat T)$, the space of constants on $\hat T$. 
	
	For $0 < r = s < 1$, we cannot invoke \eqref{eq:thm:interpolation:4}. Instead, scaling and \eqref{eq:thm:interpolation:5} imply that
	\begin{align*}
		\seminorm[W^{1,p}(\cD)]{v_h}^p &\lesssim \sum_{T \in \cT_h} |\det B_T| |(B_T^*)^{-1}|^p \seminorm[L_p(\hat T)]{\nabla \hat v_h}^p \\
		&\lesssim \sum_{T \in \cT_h} |\det B_T| |(B_T^*)^{-1}|^p \seminorm[W^{s,p}(\hat T)]{\hat v}^p \\
		&\lesssim \sum_{T \in \cT_h} |\det B_T|^{-1} |(B_T^*)^{-1}|^p |B_T^{-1}|^{-d-sp} \seminorm[W^{s,p}(T)]{v}^p \lesssim h^{(s-1)p} \seminorm[W^{s,p}(\cD)]{v},
	\end{align*}
	using quasiuniformity in the last step. From this we obtain $\norm[W^{1,p}]{e_h} \lesssim h^{(s-1)} \norm[W^{s,p}]{v}$. Interpolating between this estimate and $\norm[L^p]{e_h} \lesssim h^{s} \norm[W^{s,p}]{v}$ completes the proof.			
\end{proof}	

\section{A computational comparison of two noise generation methods in finite element approximation of an SPDE}

In this supplementary material, we provide a detailed computational comparison of two methods for noise simulation within the finite element method framework: the circulant embedding method, which necessitates noise interpolation, and a truncated numerical eigenapproximation, which does not. Since the latter is not feasibly implementable in spatial dimension greater than $d=1$ on a laptop computer, we restrict ourselves to this case. 

Noise interpolation avoids the need to compute integrals involving the incremental covariance kernel $q$. If $q$ is stationary, the circulant embedding method can be applied to sample noise interpolated on a uniform grid exactly. If, in addition, $q$ is a Mat\'ern kernel, it is known that the noise simulation cost at each time step is $\cO(\log(N_h)N_h)$, where $N_h$ is the number of nodes in the finite element mesh \cite{GKNSS18}. We therefore expect the first of the considered methods to be competitive when the noise is of Mat\'ern type.

The simulation study is performed on a linear stochastic heat equation on $\cD = (0,1)$. We consider additive noise, zero initial value and $q$ of Matérn type (Example~\ref{ex:matern}) with parameters $\nu \in (0,1)$ and $\sigma = \rho = 1$. We let the elliptic operator $A$ be given by $a(1-\Delta)$, where $\Delta$ is the Laplacian with Neumann boundary conditions and $a=0.05$. The methods are implemented in MATLAB R2023b.

\subsection{An approximation employing the circulant embedding method}

The first SPDE approximation we consider is the one analyzed in the main paper, i.e.~\eqref{eq:spde-approximation}. In this one-dimensional setting, there is no need to consider two domains and two finite element meshes. Instead, we take $\cT_{h'} = \cT_h$ to be a uniform discretization of $[0,1]$ with mesh size $h$. Then, the number of nodes $N_h$ is equal to $h^{-1} + 1 \simeq h^{-1}$. The interpolated noise $I_h \Delta W^j$ at time step $t_j$ is sampled by the circulant embedding method, for the implementation of which we refer to \cite{GKNSS18}. From Example~\ref{ex:matern}, we see that in this context the noise discretization error dominates and we should see a spatial convergence rate in Theorem~\ref{thm:main} of $1/2 + \nu - \epsilon$ for arbitrary $\epsilon > 0$. This is confirmed in Figure~\ref{fig:strong-error-ce-vs-eig}, where we have plotted Monte Carlo approximations of the strong error for fixed time step $\Delta t = 2^{-11}$, decreasing $h$ and $\nu \in \{0.3,0.5,0.7\}$, with reference solutions computed with $h=2^{-12}$ used in place of the mild solutions and the same realization of $W$ used across all values of $h$. To implement the circulant embedding method, one must first embed the covariance matrix $\bar Q$ with entries $q(x_i-x_j)$ into a circulant matrix, which can be sampled from using FFT, provided it is positive semidefinite. If it is not, it has to be extended until positive semidefiniteness is obtained. This extension grows with the smoothness parameter $\nu$ and the correlation length $\rho$. In our case, $\rho$ is very big (equal to the size of the spatial domain) but only for $\nu= 0.7$ is the cost of extending the circulant matrix non-negligible, see Figure~\ref{fig:offline-ce-vs-eig}. We have here extended the circulant matrix one row/column at a time until positive definiteness was reached. A more efficient search algorithm would decrease this cost. We note that this cost is \textit{offline}, meaning it only has to be done once for a fixed $h$, irrespective of how small we make $\Delta t$ or how many Monte Carlo samples of the SPDE approximation path we take in any subsequent step. Asymptotically, the \textit{online} cost of sampling from the extended matrix at each timestep is $\cO(\log(N_h)N_h)$. In the range of $h$ we consider, this cost was not achieved in our implementation. Instead, the cost remains relatively constant across all $h$, see Figure~\ref{fig:online-ce-vs-eig}. The cost was computed for a fixed time step $\Delta t = 2^{-9}$ using an average of times over $10$ samples of $X_{h,\Delta t}$.

\begin{figure}[ht]
	\centering
	\includegraphics[height=0.31\textheight]{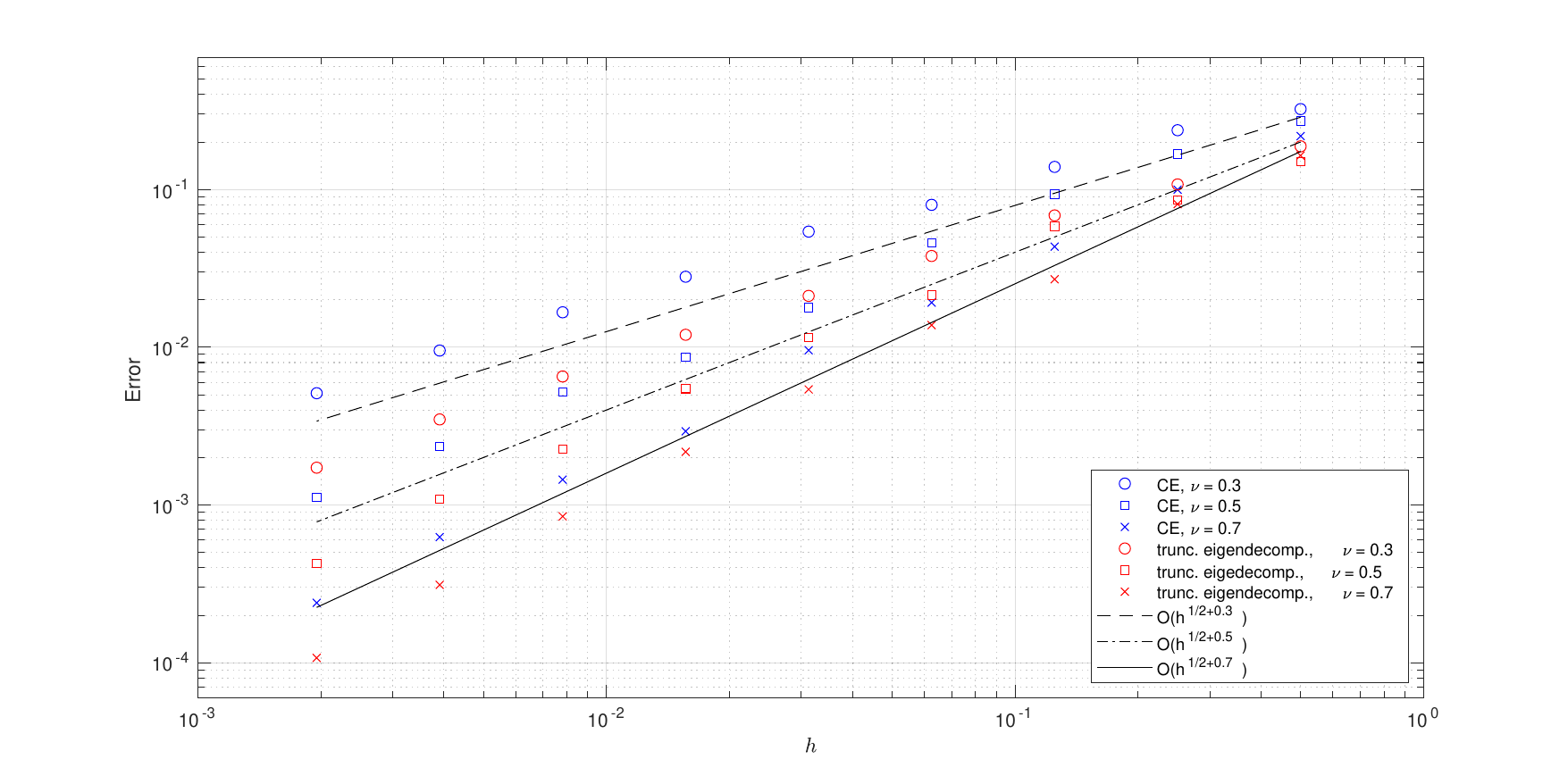}
	\centering
	\caption{Approximate strong errors for circulant embedding approximation and eigentruncation approximation. For both methods we have used a reference solution with $\Delta t = 2^{-11}, h = 2^{-12}$ and a total of $N = 10$ Monte Carlo samples, with the same realization of $W$ used for every $h$.}
	\label{fig:strong-error-ce-vs-eig}
\end{figure}

\begin{figure}[ht]
	\centering
	\includegraphics[height=0.31\textheight]{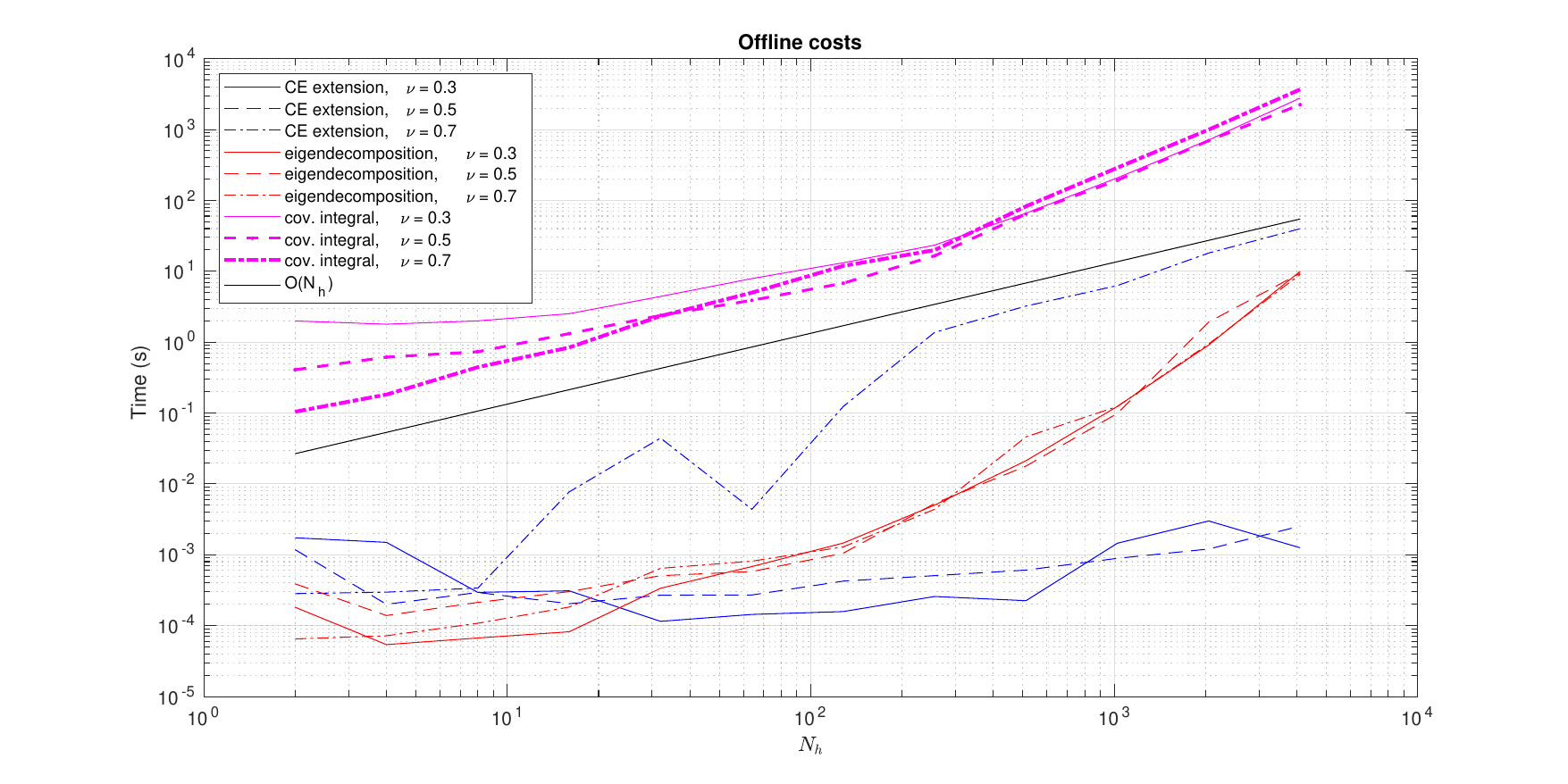}
	\caption{Offline costs for computing the covariance integrals~\eqref{eq:covariance-integrals}, the eigendecomposition and the extension for the circulant embedding matrix.}
	\label{fig:offline-ce-vs-eig}
\end{figure}

\begin{figure}[ht]
	\centering
	\includegraphics[height=0.31\textheight]{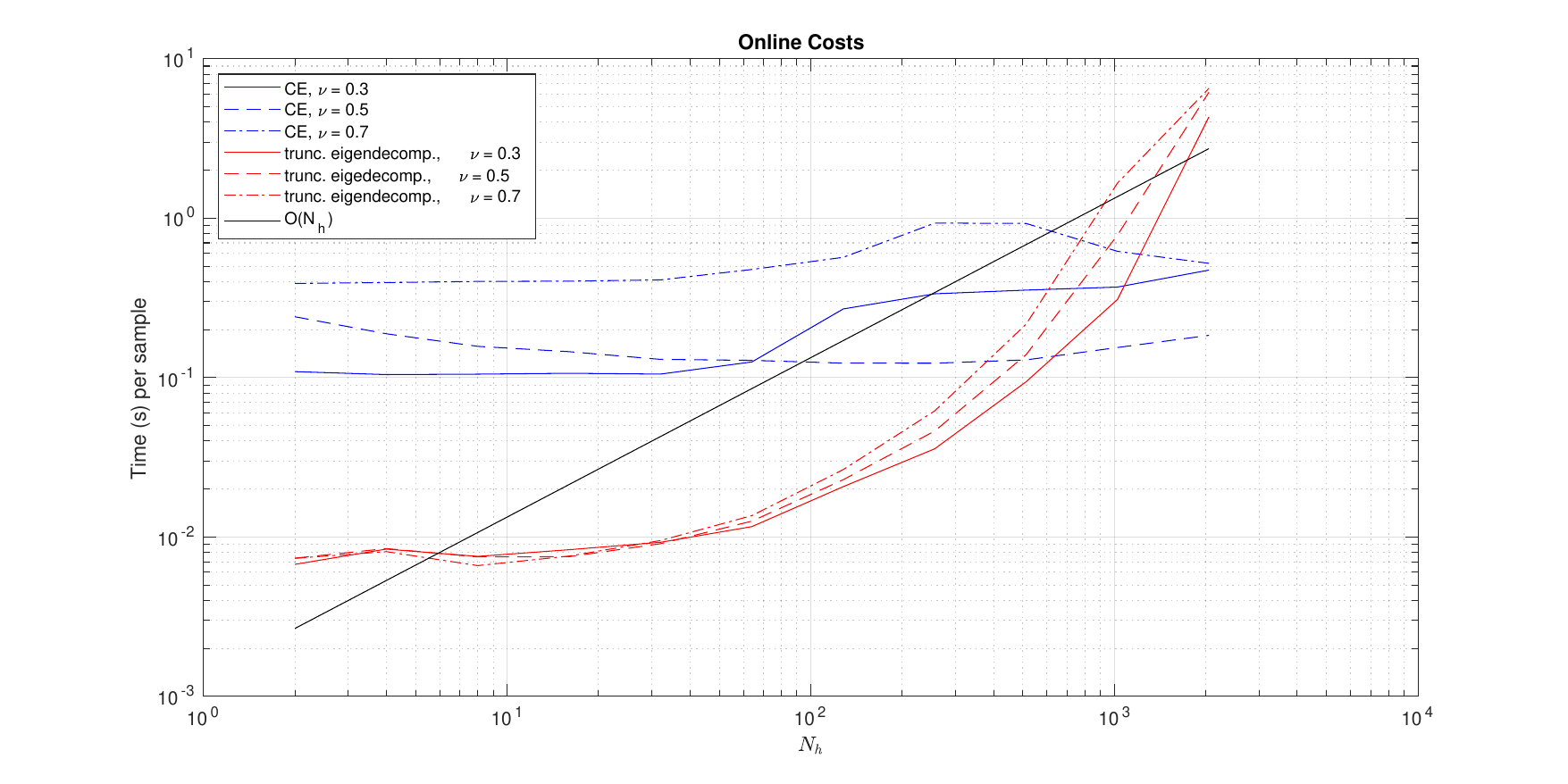}
	\caption{Average cost for obtaining one path of $X_{h,\Delta t}$ for fixed $\Delta t = 2^{-9}$ for the eigendecomposition method and the circulant embedding method.}
	\label{fig:online-ce-vs-eig}
\end{figure}

\subsection{An approximation employing a truncated numerical eigenapproximation}

The second method we consider starts by considering the same SPDE approximation as before, but without noise interpolation:
\begin{equation*}
	X_{h,\Delta t}^j - X_{h,\Delta t}^{j-1} + \Delta t A_h X_{h,\Delta t}^j = \dot{P}_h \Delta W^j, \quad j = 1, 2, \ldots 
\end{equation*}
The only computational difference is that we must sample from the covariance matrix $\tilde Q$ associated with the projected noise $\dot{P}_h \Delta W^j$, which has entries 
\begin{equation}
	\label{eq:covariance-integrals}
	\int_{\cD} \int_{\cD} q(x-y) \phi_i(x) \phi_j(y) \dd x \dd y, \quad i,j = 1, \ldots, N_h,
\end{equation}
where $\phi_i$ is the hat function associated with node $x_i$. Computing these $N_h (N_h + 1)/2$ covariance integrals is very costly when using the MATLAB function \textit{integral2}, and it is the limiting factor in this numerical experiment, see Figure~\ref{fig:offline-ce-vs-eig}. We note that we have not made an attempt to exploit the stationarity of $q$ here, which might be possible in this uniform grid setting and might decrease this cost. To sample from $\tilde Q$, we also need to compute its eigendecomposition. This is known to be of cost $\cO(N_h^3)$. This cost was not achieved in our implementation, but it is clear that the cost grows much faster than $\cO(N_h)$, see Figure~\ref{fig:offline-ce-vs-eig}, and would eventually overtake the offline cost for the circulant embedding method, even in the case that $\nu = 0.7$. From Example~\ref{ex:matern} and the proof of Theorem~\ref{thm:main} (ignoring the noise interpolation part) it can be seen that the spatial convergence rate is $1 + \nu - \epsilon$ for arbitrary $\epsilon > 0$. To compare apples with apples, therefore, we \textit{truncate} the eigenexpansion of $\dot{P}_h \Delta W^j$ at an $N \in \N$ to introduce a noise discretization error and reduce the asymptotic online computational cost per time step from $\cO(N_h^2)$ to $\cO(N N_h)$. To find out which $N$ to choose to obtain a spatial convergence rate of $1/2 + \nu - \epsilon$, one could in principle make use of the semidiscrete analysis of \cite{KLL10}. The problem is that this analysis does not take into account the smoothing of the semigroup approximation, which leads to non-sharp results. Instead, we \textit{conjecture} that if we choose $N = N_h^{\frac{1/2 + \nu}{1 + \nu}}$, we would see a spatial convergence rate of $1/2 + \nu  - \epsilon$. Based on the strong errors in Figure~\ref{fig:strong-error-ce-vs-eig}, where we have used this value of $N$, we believe this conjecture to be true. The asymptotic online computational cost per time step is then $\cO(N_h^{\frac{3/2 + 2\nu}{1 + \nu}})$, i.e., approximately $\cO(N_h^{1.62})$ for $\nu = 0.3$ and $\cO(N_h^{1.71})$ for $\nu = 0.7$. The observed costs are shown in Figure~\ref{fig:online-ce-vs-eig}.

\subsection{Conclusion}

These computational experiments indicate that the two methods converge with the same rate in space. But even in the given setting of additive noise in $d=1$, the offline costs for the truncated eigendecomposition method using projected noise completely dominate the offline costs for the circulant embedding method using interpolated noise. This is mainly due to the cost of computing the integrals $$\int_{\cD} \int_{\cD} q(x-y) \phi_i(x) \phi_j(y) \dd x \dd y.$$
Even if one disregards the cost of computing the integrals, the offline cost of computing the eigendecomposition dominates the offline cost of extending the circulant matrix for $\nu=0.3, 0.5$ and will eventually dominate for $\nu= 0.7$. And even if one disregards the offline costs, the online costs are substantially lower for the circulant embedding method for small $h$. These will dominate as $\Delta t \to 0$ and the number of Monte Carlo samples tend to $\infty$.

This is not a complete analysis of the advantages and disadvantages of the two methods.
First, sharp theoretical results in the spirit of~\cite{KLL10} for this spectral method are needed for rigorous comparisons. These are, as far as we know, not available in the literature. Second, there are several questions that need to be answered for a fair comparison. What is the offline cost for the circulant embedding method if a more efficient search algorithm is used? Can the computation of the integrals be sped up if stationarity of $q$ is taken into account? Finally, we think a fruitful future approach is to consider interpolated noise (so no integrals involving $q$ have to be computed) and then approximate and truncate the eigenbasis of the covariance matrix with entries $q(x_k,x_\ell)$ for finite element nodes $x_k,x_\ell$. This additional approximation to interpolated noise will be analyzed in future work.

\end{document}